\documentclass[12pt]{article}
\usepackage{amsmath,amsthm,amssymb}

\numberwithin{equation}{section}

\providecommand{\C}[1]{\mathcal{#1}}
\providecommand{\D}[1]{\mathbb{#1}}

\theoremstyle{plain}
\newtheorem{thm}{Theorem}

\newtheorem{lem}[thm]{Lemma}
\newtheorem{prop}[thm]{Proposition}
\theoremstyle{definition}

\newtheorem{rem}[thm]{Remark}
\newtheorem*{assumps*}{Assumptions}
\newtheorem*{nots*}{Notations}
\newtheorem*{rem*}{Замечание}
\def \ve{\varepsilon}
\begin{document}
\title
{Homogenization of boundary value problems for monotone operators
in perforated domains with rapidly oscillating  boundary
conditions of Fourier type}
\author
{ A. Piatnitski$^\flat$, V. Rybalko$^\sharp$}

\maketitle

\parskip=0.1truein
{\footnotesize
\begin{center}
$^\flat$
$^\sharp$ Narvik University College\\
Postboks 385, 8505 Narvik, Norway\\
and\\
P.N.Lebedev Physical Institute of RAS,\\
Leninski pr., 53, Moscow 117924, Russia \\
{\tt andrey@sci.lebedev.ru}
\end{center}
\begin{center}
Department of Mathematics,\\
B. Verkin Institute
for Low Temperature Physics
and Engineering (FTINT)\\
47 Lenin Ave., Kharkov 61103, Ukraine\\
{\tt  vrybalko@ilt.kharkov.ua}
\end{center}
}

\bigskip
\begin{abstract}
The paper deals with homogenization problem for nonlinear elliptic
and parabolic equations in a periodically perforated domain, a nonlinear Fourier boundary conditions being imposed on the perforation border. Under the assumptions that the studied differential equation satisfies monotonicity  and 2-growth conditions and that   the coefficient of the boundary operator is centered at each level set of unknown function, we show that the problem under consideration admits homogenization and derive the effective model.
\end{abstract}

\section{Introduction}

This paper addresses the homogenization of the boundary value
problem
\begin{equation}
\begin{cases}
-{\rm div}\, a(Du_\ve, x/\ve)+\lambda u_\ve=f\ \text{in}\ \Omega_\ve\\
a(Du_\ve, x/\ve)\cdot\nu=0\ \text{on}\ \partial \Omega\\
a(Du_\ve, x/\ve)\cdot\nu=g(u_\ve,x/\ve)\ \text{on}\ S_\ve,
\end{cases}
\label{eq1}
\end{equation}
where $\Omega_\ve$ is a bounded periodically perforated domain in
$\D{R}^N$ ($N\geq 2$), $\ve>0$ is a small parameter referred to
the perforation period. The boundary of $\Omega_\ve$ consists
of two parts, namely, the fixed outer boundary $\partial\Omega$,
and the boundary of perforations $S_\ve$. We assume that the domain is not perforated in a small (of order $\ve$) neighbourhood of $\partial\Omega$
so that the perforation boundary $S_\ve$ and $\partial\Omega$ are disjoint.
The coefficients
$a=(a_1,\dots,a_N)$ in the equation and the function $g$ in the
boundary condition on $S_\ve$ are strongly oscillating (with the
period $\ve$) functions. The boundary condition on $S_\ve$
includes, as a particular case, the inhomogeneous Neumann boundary
condition of the form $a(Du_\ve, x/\ve)\cdot\nu=\alpha(x/\ve)$ and
the Fourier one, $a(Du_\ve,
x/\ve)\cdot\nu=\beta(u_\ve,x/\ve)u_\ve$. Along with the
stationary problem (\ref{eq1}) we also consider the parabolic
problem
\begin{equation}
\begin{cases}
\partial_t u_\ve -{\rm div}\, a(Du_\ve, x/\ve)=f\ \text{in}\
\Omega_\ve\times\{t>0\}\\
a(Du_\ve, x/\ve)\cdot\nu=0\ \text{on}\ \partial \Omega\\
a(Du_\ve, x/\ve)\cdot\nu=g(u_\ve,x/\ve)\ \text{on}\ S_\ve\\
u_\ve=\tilde u\ \text{for}\ t=0.
\end{cases}
\label{eq2}
\end{equation}

The linear elliptic equations in perforated domains with the
Fourier boundary condition on the boundary of perforations were
considered, e.g., in \cite{CD}, \cite{CS}, \cite{ADH}, \cite{BCP},
\cite{OS}, \cite{OS1}, \cite{P}. It was shown that if the
coefficient in the Fourier boundary condition is small (of order
$\ve$), or the volume fraction of the holes vanishes at a certain
rate, as $\ve\to0$, then the asymptotic behaviour of solutions to
these equations is described in terms of a homogenized problem with
an additional potential. By contrast, if the volume fraction of the
holes does not vanish as the period of the structure tends to zero, then the
dissipative Fourier boundary condition forces solutions vanish.

In the problem studied in the present work the surface
measure $|S_\ve|$ tends to infinity as $\ve\to 0$. To compensate
this measure grows we assume that the average of the function
$g(u,x/\ve)$ (appearing in the boundary condition on $S_\ve$) over
the boundary of each hole is zero for any $u\in\D{R}$.

Previously,  linear
problems with the same assumptions on the coefficient in the Fourier boundary
condition were considered in \cite{BCP1}; related spectral problems
were  studied in \cite{P1},\cite{P2}. The corresponding
homogenized operator is shown to contain an additional potential, this potential is always negative.

A variational problem closely related to (\ref{eq1}) for a
functional with a bulk energy and a surface term on the
perforation boundary was studied in \cite{CP} by means of
$\Gamma$-convergence technique.

In contrast to \cite{CP} we do not assume that the problem under consideration can be written in variational form. Instead, we assume
the monotonicity of $a(\xi,y)$ and apply here the celebrated
two-scale convergence method (see, e.g. \cite{N}, \cite{A},
\cite{LNW}). This allows us to treat boundary value problems that can not be reduced to the minimization of an energy functional; for
instance, such a reduction is not possible in the case of linear function $a(\xi,y)$, $a(\xi,y)=A(y)\xi$,  with nonsymmetric matrix $A$.

Since, in general, the monotonicity assumption on $a(\xi,y)$ does
not imply the monotonicity of the problem (\ref{eq1}) (even for
large $\lambda$) we are not able to show the uniqueness result for
(\ref{eq1}). Moreover, the existence of a solution of (\ref{eq1})
holds only for sufficiently large $\lambda$ (see the discussion in
\cite{CP}), while the parabolic problem (\ref{eq2}) does have a
unique solution under certain assumptions on $a(\xi,y)$ and
$g(u,y)$.

The key difficulty in applying the two-scale convergence
theory to the homogenization of (\ref{eq1}) and (\ref{eq2}) is due
to the presence of a highly perturbed surface integral in the weak
formulations of the said problems. To pass to the limit in the
surface integral we establish a new result related to the
two-scale convergence of traces, see Proposition
\ref{convergtraces}.

The main result of this work shows that solutions $u_\ve$ of
problem (\ref{eq1}) converge as $\ve\to 0$ to a solution $U_0$ of
the homogenized problem
\begin{equation}
\begin{cases}
{\rm div}\, a^*(DU_0,U_0)+
b^*(DU_0,U_0)+|Y^*|(f-\lambda U_0)=0\ \text{in}\ \Omega\\
a^*(DU_0,U_0)\cdot\nu=g^*(U_0)\cdot\nu\ \text{on}\ \partial
\Omega.
\end{cases}
\label{homogenizedeq1}
\end{equation}
The coefficients $a^*$, $b^*$ are defined in terms of a cell
problem (see problem (\ref{celleq})) and depend both on the
coefficients $a=(a_1,\dots,a_N)$ in the equation in (\ref{eq1})
and on the function $g$ in the boundary condition on $S_\ve$. It
is interesting to observe also that
the homogenization of (\ref{eq1}) leads to the change of the
boundary condition on $\partial \Omega$ from the homogeneous
Neumann condition to a Fourier type one.

In what concerns the parabolic problem (\ref{eq2}), we show that
solutions $u_\ve$ of (\ref{eq2}) converge as $\ve\to 0$ to a
solution $U_0$ of the homogenized problem
\begin{equation}
\begin{cases}
|Y^*|\partial_t U_0-{\rm div}\, a^*(DU_0,U_0)-b^*(DU_0,U_0)
=|Y^*|f\ \text{in}\ \Omega\times \{t>0\}\\
a^*(DU_0,U_0)\cdot\nu=g^*(U_0)\cdot\nu\ \text{on}\ \partial \Omega
\\U_0=\tilde u\ \text{when}\ t=0.
\end{cases}
\label{homogenizedeq2}
\end{equation}
The analysis of (\ref{eq2}) involves the same ideas as that of
(\ref{eq1}) combined with a lower semicontinuity trick already
used in the parabolic problems in \cite{CP1}, \cite{CP2},
\cite{ClS}, \cite{PPR}.

An interesting issue in both parabolic and elliptic frameworks is
the uniqueness of a solution of the limit problem. The limit operator, although admits a priory estimates, need not be monotone even for large values of $\lambda$. The main difficulty is due to the fact that the first order term $b^*(D u,u)$ in the limit equation couples the unknown function $u$ and its gradient.\\
The uniqueness is proved only for small space dimensions and in the case when
either $a(\xi,y)$ is linear in $\xi$ or $g(u,y)$ is linear in $u$.
Without these additional assumptions it remains an open problem.

The paper is organized as follows. Section \ref{prasentationofmainresults}
is devored to problem setup and formulation of the main results.

Sections \ref{s_stationary}--\ref{section5} deal with the elliptic case.
In Section \ref{s_stationary} we prove the two-scale convergence result which relies on several technical statements. These technical statements are then justified in Sections  \ref{section4} and \ref{section5}.

Section \ref{section6} considers the parabolic case.

Finally, in Section \ref{section7} we study the properties of the homogenized
problems.

\section{Presentation of main results}
\label{prasentationofmainresults}

Let $Y$ be the unit cube $Y=[-1/2,1/2)^N$ ($N\geq 2$), and let $G$
be an open subset of $Y$ such that $\overline{G}\subset
(-1/2,1/2)^N$, with Lipschitz boundary. Set $Y^*=Y\setminus G$
and
$S=\bigcup_{m\in\D{Z}}(\partial G+m)$.

Given a bounded connected open set $\Omega\subset\D{R}^N$ with Lipschitz
boundary $\partial \Omega$, we consider the perforated domain
$\Omega_\ve$ defined by
$$
\Omega_\ve=\Omega\setminus \bigcup_{m\in I_\ve}(\ve G+m\ve),
\  I_\ve=\{m\in\D{Z}^N; Y_\ve^{(m)}\subset \Omega\},
$$
where $Y_\ve^{(m)}=(Y+m)\ve$. We have
$\partial\Omega_\ve=\partial\Omega\cup S_\ve$, where $S_\ve$ is
the boundary of perforations.

We assume that $a:\D{R}^N\times Y\to \D{R}^N$ and $g:\D{R}\times
S\to\D{R}$ satisfy
\begin{itemize}
\item[(i)\ \ ] $a(\xi,y)$ (resp. $g(u,y)$) is continuous in $\xi$
(resp. $u$), i.e. $a\in C(\D{R}^N; L^\infty(Y))$, $g\in C(\D{R};
L^\infty(S))$, and $Y$-periodic in $y$; \item[(ii)\ ] there is
$\kappa>0$ such that
\begin{equation}
(a(\xi,y)-a(\zeta,y))\cdot(\xi-\zeta)\geq \kappa |\xi-\zeta|^2\
\forall \xi,\zeta\in\D{R}^N; \label{monoton}
\end{equation}
\item[(iii)] there are constants $C_1,\dots,C_8>0$ such that
\begin{equation}
-C_1 +C_2 |\xi|^2\leq a(\xi,y)\cdot \xi,\ |a(\xi,y)|\leq C_3
|\xi|+C_4, \label{3}
\end{equation}
\begin{equation}
|g(u,y)|\leq C_5 |u| +C_6, \label{4}
\end{equation}
\begin{equation}
|g(u,y)-g(v,y)|\leq C_7 |u-v|, \label{5}
\end{equation}
\begin{equation}
|g^\prime_u(u,y)-g^\prime_u(v,y)|\leq C_8 |u-v|(1+|u|+|v|)^{-1};
\label{6}
\end{equation}
\item[(iv)\ ]
\begin{equation}
\int_{S\cap Y}g(u,y) \, {\rm d}\sigma_y=0,\ \forall u\in \D{R}.
\label{7}
\end{equation}
\end{itemize}

Let us rewrite (\ref{eq1}) in an abstract form. To this end
consider the space $X_\ve=W^{1,2}(\Omega_\ve)$ and its dual
$X_\ve^*$ with respect to the duality pairing
$\langle\,\cdot\,,\,\cdot\,\rangle_\ve$ induced by the standard
inner product in $L^2(\Omega_\ve)$.
Define the operators ${\C A}_\ve,\, {\C G}_\ve:X_\ve\to X_\ve^*$
by
\begin{multline}
\langle{\C A}_\ve(u),v \rangle_{\ve} =\int_{\Omega_\ve}a(Du,
x/\ve)\cdot D v {\rm d}x,\ \langle
{\C G}_\ve(u),v\rangle_{\ve}=
\int_{S_\ve}g(u,x/\ve)v{\rm d}\sigma, \\
\forall v\in X^\ve(=W^{1,2}(\Omega_\ve))\label{defAG}.
\end{multline}
In terms of these operators (\ref{eq1}) reads
$$
{\C A}_\ve(u_\ve)+\lambda u_\ve-{\C G}_\ve(u_\ve)=f.
$$
According to the assumptions (i)-(iii) the operator ${\C A}_\ve$
is monotone and continuous while ${\C G}_\ve$ is a compact
operator. It follows that ${\C F}_\ve(u)={\C A}_\ve(u)+\lambda
u-{\C G}_\ve(u)$ ($\lambda>0$) is a bounded continuous and
pseudo-monotone operator (recall that  ${\C F}_\ve:\,X_\ve\to
X_\ve^*$ is pseudo-monotone  if $u^{(i)}\to u$ weakly in $X_\ve$
and $\limsup_{i\to\infty}\langle{\C
F}_\ve(u^{(i)}),u^{(i)}-u\rangle_\ve\leq 0$ imply $\langle{\C
F}_\ve(u),u-v\rangle_\ve\leq \liminf_{i\to\infty}\langle{\C
F}_\ve(u^{(i)}),u^{(i)}-v\rangle_\ve$ for all $v\in X_\ve$). Then
for any $f\in L^2(\Omega)$ problem (\ref{eq1}) has a (possibly not
unique) solution $u_\ve\in X_\ve$ when $\ve\leq\ve_0$,
$\lambda\geq\lambda_0$ (where $\lambda_0,\ve_0>0$ are specified in
Theorem \ref{aprioriestimate} below) by Brezis' theorem (see,
e.g., \cite{S}, Chapter II), thanks to the following coercivity
result
\begin{thm}
\label{th0}
 \label{aprioriestimate} Under
assumptions {\rm (i)}-{\rm (iv)} there are $\lambda_0,\ve_0>0$
such that
\begin{equation}
\langle{\C A}_\ve u+\lambda u-{\C G}_\ve(u),u\rangle_\ve\geq
\kappa_1 \|u\|_{X_\ve}^2-\kappa_2, \label{coercive}
\end{equation}
when $\|u\|_{X_\ve}\geq R$, for some $\kappa_1>0$, $\kappa_2>0$
and $R>0$ independent of $\ve\leq\ve_0$ and
$\lambda\geq\lambda_0$.
\end{thm}

Under the above assumptions on the perforated domain $\Omega_\ve$
there is a bounded linear extension operator
$P_\ve:W^{1,2}(\Omega_\ve)\to W^{1,2}(\Omega)$ ($P_\ve v=v$ in
$\Omega_\ve$ for any $v\in W^{1,2}(\Omega_\ve)$) and $\|P_\ve
v\|_{W^{1,2}(\Omega)}\leq C\|v\|_{W^{1,2}(\Omega_\ve)}$, $\|P_\ve
v\|_{L^2(\Omega)}\leq C\|v\|_{L^2(\Omega_\ve)}$
 with $C$ independent of $\ve$ (see,e.g. \cite{ACPDMP}). We keep the
notation $u_\ve$ for the solution of (\ref{eq1}) extended to
$\Omega_\ve$ ($u_\ve=P_\ve u_\ve$) and study the asymptotic
behavior of $u_\ve$ as $\ve\to 0$.

 The first main result of this work is
\begin{thm}
\label{th1}
Assume that conditions {\rm (i)}-{\rm (iv)} are
satisfied and $f$ in (\ref{eq1}) belongs to $L^2(\Omega)$. Let
$\lambda_0>0$ be as in Theorem \ref{aprioriestimate}. Then for any
$\lambda\geq \lambda_0$, solutions $u_\ve$ of (\ref{eq1}) and
their derivatives $Du_\ve$ two-scale converge as $\ve\to 0$ (up to
extracting a subsequence) to $U_0(x)$ and $DU_0(x)+D_yU_1(x,y)$,
where the pair $U_0(x)$, $U_1(x,y)$ is a solution of the two-scale
homogenized problem: find $U_0(x)\in W^{1,2}(\Omega)$,
$U_1(x,y)\in L^2(\Omega; W^{1,2}_{per}(Y))$ such that
\begin{multline}
\int_\Omega\int_{Y^*} (a(DU_0+D_yU_1,y)\cdot
(D\Phi_0+D_y\Phi_1){\rm d}y{\rm d}x
\\
-\int_\Omega\int_{S\cap Y} (g(U_0,y)\Phi_1(x,y)+
g^\prime_u(U_0,y)\Phi_0U_1(x,y)){\rm d}\sigma_y{\rm d}x
\\
-\int_\Omega\int_{S\cap Y} D_x (g(U_0,y)\Phi_0)\cdot y {\rm
d}\sigma_y{\rm d}x
 -\int_\Omega
|{Y^*}|(f-\lambda U_0)\Phi_0{\rm d}x= 0, \label{varequality}
\end{multline}
for any $\Phi_0(x)\in W^{1,2}(\Omega)$, $\Phi_1(x,y)\in
L^2(\Omega; W^{1,2}_{per}(Y))$. In particular, $u_\ve$ converge
weakly in $W^{1,2}(\Omega)$ to a solution $U_0$ of the homogenized
problem (\ref{homogenizedeq1}), where $a^*(\xi,u)$, $b^*(\xi,u)$,
$g^*(u)$ are defined by
\begin{equation}
a^*(\xi,u)=\int_{Y^*}a(\xi+D_y w,y){\rm d} y,  \label{astar}
\end{equation}
\begin{equation}
b^*(\xi,u)=\int_{S\cap Y}g^\prime_u(u,y)w{\rm d} \sigma_y,
\label{bstar}
\end{equation}
\begin{equation}
g^*(u)=\int_{Y^*}g(u,y)y{\rm d}\sigma_y,  \label{gstar}
\end{equation}
and $w=w(y;\xi,u)$ is a unique (up to an additive constant)
solution of the cell problem
\begin{equation}
\begin{cases}
{\rm div}\, a(\xi+D_y w,y)=0\ \text{in}\ Y^*\\
a(\xi+D_y w,y)\cdot\nu=g(u,y)\ \text{on}\ S\cap Y\\
w\ \text{is $Y$-periodic}.
\end{cases}
\label{celleq}
\end{equation}
\end{thm}

\begin{rem}
Note that (\ref{varequality}) defines $U_1(x,y)$ modulo an
arbitrary function $\tilde U_1(x,y)\in
L^2(\Omega,W^{1,2}_{per}(Y)$ such that $U_1(x,y)=0$ for $y\in
Y^*$. This is due to the freedom in the particular choice of the
extension operators $P_\ve$.
\end{rem}

\begin{rem} The third term in (\ref{varequality}) is reduced by
integrating by parts to the boundary integral
$$
\int_\Omega\int_{S\cap Y} D_x (g(U_0,y)\Phi_0)\cdot y {\rm
d}\sigma_y{\rm d}x=\int_{\partial
\Omega}\Phi_0g^*(U_0)\cdot\nu{\rm d}\sigma_x,
$$
that leads to the boundary condition in (\ref{homogenizedeq1}).
\end{rem}

\begin{rem} In the linear case, that is when $a$ and $g$
are given by $a(\xi,y)=A(y)\xi$, $g(u,y)=\alpha(y)+u\beta(y)$, the
cell problem (\ref{celleq}) for $w$ splits into three cell
problems for $w^{(1)}$,
\begin{equation}
\begin{cases}
{\rm div}\, (A(y)(\xi+D_y w^{(1)}))=0\ \text{in}\ Y^*\\
A(y)D_y w^{(1)}\cdot\nu=-A(y)\xi\cdot\nu\ \text{on}\ S\cap Y\\
w^{(1)}\ \text{is $Y$-periodic},
\end{cases}
\label{celleq1}
\end{equation}
and $w^{(k)}$ ($k=2,3$),
\begin{equation}
\begin{cases}
{\rm div}\, (A(y)D_y w^{(k)})=0\ \text{in}\ Y^*\\
A(y)D_y w^{(k)}\cdot\nu=\delta_{2k}\beta(y)+\delta_{3k}\alpha(y)\ \text{on}\ S\cap Y\\
w^{(2)}\ \text{is $Y$-periodic},
\end{cases}
\label{celleq2}
\end{equation}
($\delta_{ij}$ is the Kronecker delta) so that $w=w^{(1)}+u
w^{(2)}+w^{(3)}$. Then the homogenized equation takes form
$$
{\rm div}A^{\rm hom} DU_0+B^{\rm hom}\cdot DU_0 +C^{\rm
hom}U_0+D^{\rm hom}+|Y^*|(f-\lambda U_0)=0,
$$
where the homogenized matrix $A^{\rm hom}$ coincides with the
classical effective matrix for the Neumann problem in perforated
domains,
$$A^{\rm hom}\xi=\int_{Y^*}A(y)(\xi+D_y w^{(1)}){\rm d}
y,
$$
and
$$
B^{\rm hom}\cdot \xi=\int_{Y^*}A(y)D_y w^{(2)}\cdot(\xi+D_y
w^{(1)}){\rm d} y,
$$
$$
C^{\rm hom}=\int_{Y^*}A(y)D_y w^{(2)}\cdot D_y w^{(2)}{\rm d} y,\
D^{\rm hom}=\int_{Y^*}A(y)D_y w^{(2)}\cdot D_y w^{(3)}{\rm d} y.
$$
Note, that $B^{\rm hom}=0$ in the selfadjoint case (when $A=A^T$).
\end{rem}

In the case of the parabolic problem (\ref{eq2}) we prove that
there is a unique solution $u_\ve$ and its asymptotic behavior in
the leading term is described by the homogenized problem
(\ref{homogenizedeq2}). Formulating the convergence result we
assume as before $u_\ve$ extended onto the whole domain $\Omega$
by means of the extension operator $P_\ve$


\begin{thm}
\label{th2} Assume that conditions {\rm (i)} - {\rm (iv)} are
satisfied. Then, if $f\in L^2((0,T)\times\Omega)$ and $\tilde u\in
L^2(\Omega)$, there is a unique solution of problem (\ref{eq2})
and it converges weakly in $L^2(0,T;W^{1,2}(\Omega))$ as $\ve\to
0$ (up to extracting a subsequence) to a solution $U_0$ of the
homogenized problem (\ref{homogenizedeq2}), where $a^*$, $b^*$,
$g^*$ are defined by (\ref{astar}), (\ref{bstar}), (\ref{gstar}),
(\ref{celleq}).
\end{thm}

\section{Proof of the convergence result for the stationary problem}
\label{s_stationary}

It follows from Theorem \ref{aprioriestimate} that
$\|u_\ve\|_{W^{1,2}(\Omega)}\leq C$, where $C$ is independent of
$\ve$. Therefore, up to extracting a subsequence,
\begin{equation}
u_\ve\to U_0(x)\ \text{two-scale}, \label{9}
\end{equation}
\begin{equation}
Du_\ve\to DU_0(x)+D_yU_{1}(x,y)\ \text{two-scale}. \label{10}
\end{equation}

Show that the pair $(U_0,\ U_1)$ solves (\ref{varequality}). To
this end we chose arbitrary functions $V_0(x)\in
C^{\infty}(\overline{\Omega})$, $V_1(x,y)\in
C^{\infty}(\overline{\Omega}\times \overline{Y})$ with $V_1(x,y)$
being $Y$-periodic in $y$, set $v_\ve =V_0(x)+\ve V_{1}(x,x/\ve)$,
and substitute the test function $w_\ve =u_\ve-v_\ve$ in the weak
formulation of (\ref{eq1}),
\begin{equation}
\int_{\Omega_\ve}(a(Du_\ve, x/\ve)\cdot D w_\ve +\lambda u_\ve
w_\ve){\rm d}x-\int_{S_\ve}g(u_\ve,x/\ve)w_\ve{\rm
d}\sigma=\int_{\Omega_\ve}fw_\ve{\rm d}x. \label{weakform}
\end{equation}
In view of the monotonicity assumption (\ref{monoton}) we then
have from (\ref{weakform}),
\begin{multline}
\int_{\Omega_\ve}(a(Dv_\ve, x/\ve)\cdot D(u_\ve-v_\ve) +\lambda
v_\ve (u_\ve-v_\ve)){\rm
d}x-\int_{S_\ve}g(u_\ve,x/\ve)(u_\ve-v_\ve){\rm
d}\sigma\\-\int_{\Omega_\ve}f(u_\ve-v_\ve){\rm d}x\leq 0.
\label{basicinequality}
\end{multline}
Since $Dv_\ve =DV_0(x)+D_y V_{1}(x,x/\ve)+\ve D_x V_{1}(x,x/\ve)$,
by using (i) and (\ref{3}) one easily shows that $\chi_\ve
a(Dv_\ve, x/\ve)\to \chi(y) a(DV_0(x)+D_yV_1(x,y),y)$ in the
strong two-scale sense, where $\chi_\ve$, $\chi$ are the
characteristic functions of $\Omega_\ve$ and $Y^*$, respectively.
This allows to pass to the limit in the first term of l.h.s. of
(\ref{basicinequality}) to get
\begin{multline}
\int_{\Omega_\ve}(a(Dv_\ve, x/\ve)\cdot D(u_\ve-v_\ve) +\lambda
v_\ve (u_\ve-v_\ve)){\rm d}x\to \\
\int_\Omega\left(\int_{Y^*} (a(DV_0+D_yV_1,y)\cdot
(DU_0+D_yU_1-DV_0-D_yV_1) + \lambda V_0 (U_0-V_0)){\rm
d}y\right){\rm d}x; \label{firstterm}
\end{multline}
also, the limit transition in the last term in l.h.s. of
(\ref{basicinequality}) yields
\begin{equation}
\int_{\Omega_\ve}f(u_\ve-v_\ve){\rm d}x\to
\int_\Omega\left(\int_{Y^*}f(U_0-V_0){\rm d}y\right){\rm d}x.
\label{lastterm}
\end{equation}
Finally, passing to the limit in the middle term we get
\begin{multline}
\int_{S_\ve}g(u_\ve, x/\ve)(u_\ve-v_\ve){\rm d}\sigma\to
\\
\int_\Omega\left(\int_{S\cap Y} g(U_0,y)(D(U_0-V_0)\cdot y+
U_1(x,y)-V_1(x,y)){\rm d}\sigma_y\right){\rm d}x\\
+\int_\Omega\left(\int_{S\cap Y}
g^\prime_u(U_0,y)(U_0-V_0)(DU_0\cdot y+ U_1(x,y)){\rm
d}\sigma_y\right){\rm d}x. \label{middleterm}
\end{multline}
The most nontrivial point is to obtain (\ref{middleterm}). The
proof of (\ref{middleterm}) is presented in full details through
Sections \ref{section4}, \ref{section5} and is based on the
following result, which is of an interest itself,

\begin{prop}
\label{convergtraces} Assume that $q(x,y)\in
C(\Omega;\L^\infty(S))$ satisfies
\begin{itemize}
\item[{\rm (a)}] $|q(x,y)-q(x^{\prime},y)|\leq C|x-x^\prime|$ with
$C>0$ independent of $x,x^\prime \in\Omega$
 and $y\in S$;
\item[{\rm (b)}] $q(x,y)$ is $Y$-periodic in $y\in S$; \item[{\rm
(c)}] $\int_{Y\cap S}q(x,y){\rm d}\sigma_y=0$ for all
$x\in\Omega$,
\end{itemize}
then for any sequence $w_\ve\in W^{1,2}(\Omega)$ such that
\begin{equation} w_\ve(x)\to W_0(x),\ Dw_\ve(x)\to
DW_0(x)+D_yW_1(x,y) \ \text{\rm two scale as}\ \ve\to 0.
\label{convergwve}
\end{equation}
we have
\begin{equation}
\int_{S_\ve}q(x,x/\ve)(w_\ve-\bar w_\ve){\rm d}\sigma\to
\int_\Omega \int_{Y\cap S}q(x,y)(DW_0\cdot y+W_1(x,y)){\rm
d}\sigma_y {\rm d}x. \label{traces}
\end{equation}
Here and in what follows we use the notation $\bar w_\ve$ for the
piecewise constant function obtained by averaging over the cells
$Y_\ve^{(m)}$,
\begin{equation}
\bar w_\ve(x)=\frac{1}{\ve^N}\int_{Y_\ve^{(m)}}w_\ve(y){\rm d}y, \
\text{for}\ x\in Y_\ve^{(m)}. \label{aver}
\end{equation}
\end{prop}

Thus (\ref{basicinequality})-(\ref{middleterm}) yield
\begin{multline}
\int_\Omega\left(\int_{Y^*} (a(DV_0+D_yV_1,y)\cdot
(DU_0+D_yU_1-DV_0-D_yV_1) + \lambda V_0 (U_0-V_0)){\rm
d}y\right){\rm d}x
\\
-\int_\Omega\left(\int_{S\cap Y} g(U_0,y)(D(U_0-V_0)\cdot y+
U_1(x,y)-V_1(x,y)){\rm d}\sigma_y\right){\rm d}x\\
-\int_\Omega\left(\int_{S\cap Y}
g^\prime_u(U_0,y)(U_0-V_0)(DU_0\cdot y+ U_1(x,y)){\rm
d}\sigma_y\right){\rm d}x
\\
-\int_\Omega\left(\int_{Y^*}f(U_0-V_0){\rm d}y\right){\rm d}x\leq
0, \label{varinequality}
\end{multline}
By an approximation argument, using (i)-(iv) we see that
(\ref{varinequality}) holds for any $V_0\in W^{1,2}(\Omega)$ and
$V_1\in L^2(\Omega;W^{1,2}_{per}(Y))$. Now, choosing
$V_0=U_0\pm\tau\Phi_0$, $V_1=U_1\pm\tau\Phi_1$, ($\tau>0$),
dividing (\ref{varinequality}) by $\tau$ and passing to the limit
as $\tau\to 0$, we obtain the two-scale homogenization problem
(\ref{varequality}).\hfill$\square$

\bigskip

Let us clarify details in the final part of the above proof when
passing from smooth $V_0$ and $V_1$ to arbitrary functions $V_0\in
W^{1,2}(\Omega)$ and $V_1\in L^2(\Omega;W^{1,2}_{per}(Y))$ in
(\ref{varinequality}). For the for the first term in the l.h.s.
this transition is justified by Nemytskii's theorem (see, e.g.,
\cite{S}, Chapter II); and it is a trivial task for the last term.
The second and third terms, corresponding to the limiting
functional $M(U_0,U_1,V_0,V_1)$ in (\ref{middleterm}), require
more attention. Let us rewrite $M(U_0,U_1,V_0,V_1)$ as
\begin{multline}
M(U_0,U_1,V_0,V_1)= \int_\Omega (g^*(U_0)\cdot
D(U_0-V_0)+(U_0-V_0)(g^*)^\prime(U_0)\cdot DU_0) {\rm d}x \\
+ \int_\Omega\int_{Y^*} D_y \Theta(y;U_0)\cdot
D_y(U_1(x,y)-V_1(x,y)){\rm d}y{\rm d}x\\
 +
 \int_\Omega\int_{Y^*}(U_0-V_0)
 D_y \Theta^\prime_u(y;U_0)\cdot
D_yU_1(x,y){\rm d}y {\rm d}x, \label{middletermrepresentation}
\end{multline}
where $(g^*)^\prime$ denotes the derivative of $g^*$, and
$\Theta(y;u)$ is a solution of the problem
\begin{equation}
\begin{cases}
\Delta_y\Theta=0\ \text{in}\ Y^*\\
\frac{\partial\Theta}{\partial \nu}=g(u,y)\ \text{on}\ S\cap Y\\
\Theta\ \text{is $Y$-periodic}.
\end{cases}
\label{celltheta}
\end{equation}
It follows from the assumptions (iii), (iv) that (\ref{celltheta})
has a unique (modulo an additive constant) solution $\Theta(y;u)$,
and $\Theta$ depends regularly on the parameter $u$, more
precisely,
\begin{equation}
\|D_y\Theta(\,\cdot\,;u)\|_{L^2(Y^*)}\leq C(|u|+1), \label{Theta1}
\end{equation}
\begin{equation}
\|D_y\Theta(\,\cdot\,;u)-D_y\Theta(\,\cdot\,;v)\|_{L^2(Y^*)}\leq
C|u-v|,\label{Theta2}
\end{equation}
\begin{equation}
\|D_y\Theta^\prime_u(\,\cdot\,;u)-
D_y\Theta^\prime_u(\,\cdot\,;v)\|_{L^2(Y^*)}\leq
C|u-v|(1+|u|+|v|)^{-1}, \label{Theta3}
\end{equation}
where $C$ does not depend on $u$, $v$. All these properties are
demonstrated similarly, e.g., we show (\ref{Theta1}) by using
(\ref{4}), (\ref{7}) and the Poincar\'e inequality
(\ref{poincare_per}) in $W^{1,2}_{per}(Y^*)$ (see Sec.
\ref{section6} ),
$$
\Bigl|\int_{Y^*}D_y\Theta\cdot D_y \Theta{\rm d} y\Bigr|=
\Bigl|\int_{S\cap Y}g(u,y)
\Bigl(\Theta-\frac{1}{|Y^*|}\int_{Y^*}\Theta\,{\rm d}
y\Bigr)\,{\rm d} y\Bigr|\leq C(|u|+1)\|D_y\Theta\|_{L^2(Y^*)}.
$$
The bounds (\ref{Theta1}) -
(\ref{Theta3}) in
conjunction with assumptions (\ref{4}) - (\ref{6})  imply

\begin{prop}
\label{continuityofM}
 The functional $M(U_0,U_1,V_0,V_1)$ defined by
(\ref{middletermrepresentation}) (or, equivalently, by the r.h.s.
of (\ref{middleterm})) is continuous in $W^{1,2}(\Omega)\times
L^{2}(\Omega;W^{1,2}_{per}(Y^*))\times W^{1,2}(\Omega)\times
L^{2}(\Omega;W^{1,2}_{per}(Y^*))$.
\end{prop}

\section{Auxiliary results and proof of Theorem \ref{aprioriestimate}}
\label{section4}

{\bf 1}({\it Some inequalities}). Recall the classical
inequalities in Sobolev spaces,
\begin{equation}
\int_{S\cap Y}\bigl|v-\int_{Y} v{\rm d}x\bigr|^2{\rm d}\sigma\leq
C \int_{Y}|Dv|^2 {\rm d}x,\ \forall\ v\in W^{1,2}(Y)\ \text{(the
Poincar\'e inequality)}, \label{poincare}
\end{equation}
\begin{equation}
\int_{S\cap Y}|v|^2{\rm d}\sigma\leq C\int_{Y}(|v|^2+|Dv|^2) {\rm
d}x,\ \forall\ v\in W^{1,2}(Y)\ \text{(the trace inequality)}.
 \label{trace}
\end{equation}
By an easy scaling argument (\ref{poincare}), (\ref{trace}) lead
to the inequalities
\begin{equation} \int_{S_\ve}|v_\ve-\bar
v_\ve|^2{\rm d}\sigma\leq C\ve \int_{\Omega}|Dv_\ve|^2 {\rm d}x,
\label{rescaledpoincare}
\end{equation}
\begin{equation}
\int_{S_\ve}|v_\ve|^2{\rm d}\sigma\leq C\ve^{-1}\biggl(
\int_{\Omega}|v_\ve|^2 {\rm d}x +\ve^2 \int_{\Omega}|Dv_\ve|^2
{\rm d}x\biggr) \label{rescaledtrace},
\end{equation}
for any $v_\ve\in W^{1,2}(\Omega)$, where $\bar v_\ve$ stands for
piecewise constant function obtained by averaging over each cell
$Y_\ve^{(m)}$ (cf. (\ref{aver})), and $C$ depends only on $S$. We
also will make use of the following inequality, which is a simple
consequence of Jensen's inequality, for any $r\geq 1$,
\begin{equation}
\int_{S_\ve}|\bar v_\ve|^r{\rm d}\sigma\leq C\ve^{-1}
\int_{\Omega}|v_\ve|^r {\rm d}x, \label{rescaledtracemean}
\end{equation}
where $C>0$ is independent of $r$ and $v_\ve$.

\bigskip
\noindent {\bf 2}({\it An asymptotic representation for surface
integral in} (\ref{basicinequality}).
To pass to the limit as $\ve \to 0$  in the surface
integral in (\ref{basicinequality}) we use

\begin{lem} Let $u_\ve, w_\ve\in {W^{1,2}(\Omega)}$, then
\begin{multline}
\int_{S_\ve}g(u_\ve,x/\ve)w_\ve{\rm d}x=
\int_{S_\ve}g(\bar
u_\ve,x/\ve)(w_\ve-\bar w_\ve){\rm d}\sigma \\
+ \int_{S_\ve}g^\prime_u(\bar u_\ve,x/\ve)\bar w_\ve (u_\ve-\bar
u_\ve){\rm d}\sigma +\varrho_\ve, \label{regularrepresentation}
\end{multline}
and
\begin{equation}
|\varrho_\ve|\leq C\bigl(\ve+
(\ve\|w_\ve\|_{L^2(\Omega)})^{2/(N+2)} \bigr)
(\|w_\ve\|_{W^{1,2}(\Omega)}^2 +\|u_\ve\|_{W^{1,2}(\Omega)}^2).
\label{regularrepresentationremainder}
\end{equation}
\label{techlem1}
\end{lem}

\begin{proof}
We have,
\begin{multline*}
g(u_\ve,x/\ve)w_\ve=
g(\bar u_\ve,x/\ve)(w_\ve-\bar w_\ve)
+
(g(u_\ve,x/\ve) -(g(\bar u_\ve,x/\ve))
(w_\ve-\bar w_\ve)\\
+(g(u_\ve,x/\ve)-g(\bar u_\ve,x/\ve)) \bar w_\ve + g(\bar
u_\ve,x/\ve)\bar w_\ve,
\end{multline*}
therefore (in view of  (\ref{7}))
\begin{multline*}
\int_{S_\ve}g(u_\ve,x/\ve)w_\ve{\rm d}\sigma = \int_{S_\ve} g(\bar
u_\ve,x/\ve)(w_\ve-\bar w_\ve)
{\rm d}\sigma \\
+ \int_{S_\ve}
(g(u_\ve,x/\ve) -(g(\bar
u_\ve,x/\ve))(w_\ve-\bar w_\ve){\rm d}\sigma \\
+ \int_{S_\ve}(g(u_\ve,x/\ve)-g(\bar u_\ve,x/\ve))\bar w_\ve {\rm
d}\sigma=I_1+I_2+I_3.
\end{multline*}
The term $I_2$ gives vanishing contribution when $\ve \to 0$.
 Really, by (\ref{5}) and (\ref{rescaledpoincare}),
\begin{equation}
\label{otsenkaI2}
|I_2|\leq C\int_{S_\ve}|u_\ve-\bar u_\ve||w_\ve
-\bar w_\ve | {\rm d}\sigma
\leq C\ve\|Du_\ve\|_{L^2(\Omega)} \|Dw_\ve\|_{L^2(\Omega)}.
\end{equation}
The term $I_3$ can be written as
\begin{multline*}
I_3=\int_0^1{\rm d} t \int_{S_\ve}(g^{\prime}_u(\bar
u_\ve+t(u_\ve-\bar u_\ve),x/\ve)-g^{\prime}_u(\bar u_\ve,x/\ve))
\bar w_\ve(u_\ve-\bar u_\ve){\rm d}\sigma \\
+ \int_{S_\ve}g^\prime_u(\bar u_\ve,x/\ve)\bar w_\ve (u_\ve-\bar
u_\ve) {\rm d}\sigma=\tilde I_3+ \int_{S_\ve}g^\prime_u(\bar
u_\ve,x/\ve)\bar w_\ve (u_\ve-\bar u_\ve) {\rm d}\sigma
\end{multline*}
By using (\ref{6}) we get
$$
|\tilde I_3|\leq C\sup_{0\leq t\leq 1}
\int_{S_\ve}\frac{t|u_\ve-\bar u_\ve|^2|\bar w_\ve|} {1+|\bar
u_\ve|+|\bar u_\ve+t(u_\ve-\bar u_\ve)|}{\rm d}\sigma,
$$
which yields after applying the Holder inequality,
\begin{multline*}
|\tilde I_3|\leq  C\sup_{0\leq t\leq 1}
\int_{S_\ve}\frac{t|u_\ve-\bar u_\ve|^2|\bar
w_\ve|}{1+t|u_\ve-\bar u_\ve|}{\rm d}\sigma \leq
C\biggl(\int_{S_\ve}|\bar w_\ve|^{q}{\rm d}\sigma\biggr)^{1/q}
\\
\times\sup_{0\leq t\leq 1} \biggl(\int_{S_\ve}|u_\ve-\bar
u_\ve|^2\frac{t^{q^\prime}|u_\ve-\bar
u_\ve|^{2q^\prime-2}}{(1+t|u_\ve-\bar u_\ve|)^{q^\prime}}{\rm
d}\sigma\biggr)^{1/q^\prime},
\end{multline*}
where $q^\prime=q/(q-1)$ and $q=2(N+2)/N$. Note that the embedding
$W^{1,2}(\Omega)\subset L^q(\Omega)$ is compact, moreover one has
(see, e.g., \cite{L})
\begin{equation}
\exists C>0\ \text{such that}\ \|u\|_{L^q(\Omega)}\leq C
\|u\|_{W^{1,2}(\Omega)}^{2/q}\|u\|_{L^2(\Omega)}^{4/(Nq)}\ \forall
u\in W^{1,2}(\Omega). \label{interpolinequality}
\end{equation}
Since $1<q^\prime<2$, we have
$$\frac{t^{q^\prime}|u_\ve-\bar
u_\ve|^{2q^\prime-2}}{(1+t|u_\ve-\bar u_\ve|)^{q^\prime}}\leq
\frac{t^{2q^\prime-2}|u_\ve-\bar
u_\ve|^{2q^\prime-2}}{(1+t|u_\ve-\bar u_\ve|)^{2q^\prime-2}}
\frac{t^{2-q^\prime}}{(1+t|u_\ve-\bar u_\ve|)^{2-q^\prime}}\leq
1
$$
for any $0\leq t\leq 1$.
Therefore, by using (\ref{rescaledpoincare}),
(\ref{rescaledtracemean}) and (\ref{interpolinequality}) we get
\begin{multline}
\label{otsenkaI3}
|\tilde I_3| \leq
C\ve^{-1/q-1/q^\prime+2/q^\prime}
\|w_\ve\|_{L^{q}(\Omega)}\,\|Du_\ve\|^{2/q^\prime}_{L^2(\Omega)}\\
\leq C \ve^{2/(N+2)}
\|w_\ve\|_{W^{1,2}(\Omega)}^{2/q}\|w_\ve\|_{L^2(\Omega)}^{4/(Nq)}
\|Du_\ve\|^{2/q^\prime}_{L^2(\Omega)}\\
\leq C(\ve\|w_\ve\|_{L^2(\Omega)})^{2/(N+2)}
(\|w_\ve\|_{W^{1,2}(\Omega)}^{2}+\|Du_\ve\|^{2}_{L^2(\Omega)}),
\end{multline}
where we have used also the Young inequality. Bounds
(\ref{otsenkaI3}) and (\ref{otsenkaI3}) yield
(\ref{regularrepresentationremainder}) (since $|\varrho_\ve|\leq
|I_2|+|\tilde I_3|$). Lemma is proved.\end{proof}

The proof of the next technical result is
similar to Lemma \ref{techlem1} (and left to the reader).

\begin{lem} If $u_\ve,\ u_\ve^{(1)}\in W^{1,2}(\Omega)$,
$v_\ve\in L^\infty(\Omega)\cap W^{1,2}(\Omega)$, then setting
$w_\ve=u_\ve-u_\ve^{(1)}$ we have
\begin{equation*}
\biggl| \int_{S_\ve}(g(\bar u_\ve,x/\ve)-g(\bar
u_\ve^{(1)},x/\ve))(u_\ve-v_\ve-\bar u_\ve+\bar v_\ve){\rm
d}\sigma\biggr|  \leq C  \|w_\ve\|_{L^2(\Omega)}
\|D(u_\ve-v_\ve)\|_{L^2(\Omega)},
\end{equation*}
\begin{equation*}
\biggl| \int_{S_\ve}(g^\prime_u(\bar u_\ve,x/\ve)\bar u_\ve-
g^\prime_u(\bar u_\ve^{(1)},x/\ve)\bar u_\ve^{(1)})(u_\ve-\bar
u_\ve){\rm d}\sigma\biggr|  \leq C \|w_\ve\|_{L^2(\Omega)}
\|Du_\ve\|_{L^2(\Omega)},
\end{equation*}
\begin{equation*}
\biggl| \int_{S_\ve}(g^\prime_u(\bar u_\ve,x/\ve)- g^\prime_u(\bar
u_\ve^{(1)},x/\ve)))\bar v_\ve(u_\ve-\bar u_\ve){\rm
d}\sigma\biggr| \leq C \|w_\ve\|_{L^2(\Omega)}
\|v\|_{L^\infty(\Omega)}\|Du_\ve\|_{L^2(\Omega)}.
\end{equation*}
\label{techlem2}
\end{lem}

\bigskip
\noindent {\bf 3}({\it Proof of Theorem} \ref{aprioriestimate}).
Assume by contradiction that there are sequences $\ve_k\to 0$,
$\lambda_k\to+\infty$ and $u_k\in W^{1,2}(\Omega_{\ve_k})$ such
that $\|u_k\|_{X_{\ve_k}}\to\infty$,
$$
\langle {\C A}_{\ve_k} (u_k),u_k \rangle_{\ve_k} +\lambda_k
\langle u_k,u_k\rangle_{\ve_k}- \langle{\C
G}_\ve(u_k),u_k\rangle_{\ve_k}\leq \delta_k
\|u_k\|^2_{X_{\ve_k}}$$ and $\delta_k\to 0$.
In
view of the definition of ${\C A}_{\ve}$ and ${\C G}_{\ve}$ this
implies that
\begin{equation*}
\int_{\Omega_{\ve_k}}(a(Dv_k, x/\ve)\cdot D v_k +\lambda_k
|v_k|^2){\rm d}x\leq \int_{S_\ve}g(v_k,x/\ve)v_k {\rm d}\sigma
+\delta_k\|v_k\|^2_{W^{1,2}(\Omega)} {\rm d}x,
\end{equation*}
where $v_k=P_{\ve_k}u_k$ is the extension of $u_k$ onto $\Omega$.
By using
(\ref{3}) and the properties of the extension operator $P_\ve$ we
then get, setting $w_k=v_k/\|v_k\|_{W^{1,2}(\Omega)}$,
\begin{equation}
\gamma\int_{\Omega}|Dw_k|^2{\rm d}x + \lambda_k
\int_{\Omega_{\ve_k}} |w_k|^2{\rm d}x
\leq\frac{1}{\|v_k\|_{W^{1,2}(\Omega)}}
\int_{S_{\ve_k}}g(v_k,x/\ve)w_k {\rm d}\sigma+\tilde \delta_k,
\label{A16}
\end{equation}
with some $\gamma>0$, where $\tilde
\delta_k=\delta_k+C/\|v_k\|^2_{W^{1,2}(\Omega)}\to 0$. Now write
\begin{multline}
\int_{S_{\ve_k}}g(v_k,x/\ve_k)w_k{\rm
d}\sigma=\int_{S_{\ve_k}}(g(v_k,x/{\ve_k})-g(\bar v_k,x/{\ve_k}))
w_k{\rm d}\sigma\\
+\int_{S_{\ve_k}}g(\bar v_k,x/\ve_k)(w-\bar w_k){\rm
d}\sigma=I_1+I_2, \label{bp1}
\end{multline}
where we have used (\ref{7}). We have, by (\ref{5}) and
(\ref{rescaledpoincare}),
\begin{multline}
|I_1|\leq C\int_{S_{\ve_k}}|v_k-\bar v_k| |w_k|{\rm d}\sigma \leq
C{\ve_k}^{1/2}\Bigl(\int_{\Omega}|Dv_k|^2 {\rm
d}x\Bigr)^{1/2}\Bigl(\int_{S_{\ve_k}}|w_k|^{2} {\rm
d}x\Bigr)^{1/2}
\\
\leq C \|Dv_k\|_{L^2(\Omega)}(\|w_k\|_{L^{2}(\Omega)}
+\ve_k\|Dw_k\|_{L^{2}(\Omega)}) \label{bp2}
\end{multline}
Similarly, by (\ref{4}) and (\ref{rescaledtracemean}),
\begin{equation}
|I_2|\leq C\int_{S_{\ve_k}}|w_k-\bar w_k| (|\bar v_k|+1) {\rm
d}\sigma \leq C \|Dw_k\|_{L^2(\Omega)}(\|v_k\|_{L^{2}(\Omega)}+1).
\label{bp3}
\end{equation}
Thus
\begin{equation}
\gamma\|Dw_k\|_{L^2(\Omega)}^2+ \lambda_k
\|w_k\|^2_{L^2(\Omega_{\ve_k})} \leq
C(\|w_k\|_{L^2(\Omega)}+\ve_k)+\tilde\delta_k,
\label{aprioriestimate1}
\end{equation}
where we have used the fact that $\|w_k\|_{W^{1,2}(\Omega)}=1$.
Therefore $\|w_k\|^2_{L^2(\Omega_{\ve_k})}\to 0$.

Due to the compactness of the embedding $W^{1,2}(\Omega)\subset
L^2(\Omega)$, up to a subsequence, $w_k\to w$ strongly in
$L^2(\Omega)$. On the other hand, according to the structure of
perforated domains $\Omega_\ve$,
$$
\int_{\Omega_{\ve_k}}w_kv{\rm d} x\to |Y^*|\int_{\Omega}wv {\rm d}
x\ \text{for any} \ v\in L^{2}(\Omega).
$$
By taking $v=w$ we get $w=0$ (since
$\|w_k\|_{L^2(\Omega_{\ve_k})}\to 0$) so that
$\|w_k\|_{L^2(\Omega)}\to 0$ . Then (\ref{aprioriestimate1})
yields $\gamma\|Dw_k\|_{L^2(\Omega)}\to 0$ and consequently
$\|w_k\|_{W^{1,2}(\Omega)}\to 0$, that is a contradiction.\hfill
$\square$

\bigskip As a byproduct of the above proof we have by (\ref{bp1}),
(\ref{bp2}), (\ref{bp3}), for any $u,v\in W^{1,2}(\Omega)$
\begin{equation}
|\langle{\C G}_\ve(u),v\rangle_\ve| \leq
C(\|u\|_{W^{1,2}(\Omega)}\|v\|_{L^2(\Omega)}+
\|v\|_{W^{1,2}(\Omega)}(\|u\|_{L^2(\Omega)}+1)+
\ve\|u\|_{W^{1,2}(\Omega)}\|v\|_{W^{1,2}(\Omega)}),
\label{boundforG1}
\end{equation}
where $C$ is independent of $\ve$. In particular,
\begin{equation}
\|{\C G}_\ve(u)\|_{X^*_\ve}\leq C(\|u\|_{X_\ve}+1), \forall u\in
X_\ve. \label{boundforG}
\end{equation}
Then we have, possibly modifying $\kappa_2$ in (\ref{coercive}),
\begin{equation}
\text{(\ref{coercive}) holds true for all}\ u_\ve\in X_\ve,
\label{coercivity}
\end{equation}
when $\ve\leq\ve_0$, $\lambda\geq \lambda_0$.

\section{Limit transition in the surface term and proof of Proposition \ref{convergtraces}}
\label{section5}

\noindent {\bf 1}({\it Proof of Proposition} \ref{convergtraces}).
Let $\Omega^\prime$ be a subdomain of $\Omega$ such that
$\overline{\Omega^\prime}\subset\Omega$, and let us define the
linear functional $b_\ve$ on $W^{1,2}(\Omega)$ by
\begin{equation}
b_\ve w_\ve =\int_{S_\ve^\prime}q(x,x/\ve)(w_\ve-\bar w_\ve)\,
{\rm d}\sigma. \label{introperator}
\end{equation}
where
$S_\ve^\prime=\bigcup_{m:\,Y_\ve^{(m)}\,\cap\,\Omega^\prime\not=\emptyset}S_\ve\cap
Y_\ve^{(m)}$. Clearly, $S_\ve^\prime\subset S_\ve$.

\bigskip

\noindent {\bf Step 1}(weak convergence of $b_\ve$). Let us show
that
\begin{equation}
\|b_\ve\|\leq C\ \text{with}\  C \ \text{independent of}\ \ve,
\label{boundedness}
\end{equation}
\begin{equation}
b_\ve w\to \int_{\Omega^\prime} \int_{Y\cap S} q(x,y) Dw(x)\cdot
y\,{\rm d}\sigma_y{\rm d}x\ \text{\rm weakly, as}\ \ve\to 0.
\label{traces1}
\end{equation}
We have by (\ref{rescaledpoincare}),
\begin{multline*}
|b_\ve w_\ve|\leq C \int_{S_\ve^\prime}|w_\ve-\bar w_\ve|{\rm
d}\sigma\leq C\ve^{-1/2}\Bigl(\int_{S_\ve^\prime}|w_\ve-\bar
w_\ve|^2{\rm d}\sigma\Bigr)^{1/2} \leq
C\|w_\ve\|_{W^{1,2}(\Omega)}.
\end{multline*}
Now chose an arbitrary $w$  from the dense (in $W^{1,2}(\Omega)$)
set $C^2(\overline{\Omega})$. We have
\begin{multline*}
b_\ve w=\sum_{m} \int_{S_\ve^\prime\cap Y^{(m)}_\ve}q(x,x/\ve) (D
w(x_\ve^{(m)})\cdot (x-x_\ve^{(m)})+O(\ve^2))\,{\rm d}\sigma
\\
=\sum_{m} \int_{S_\ve^\prime\cap Y^{(m)}_\ve}q(x_\ve^{(m)},x/\ve)
D w(x_\ve^{(m)})\cdot (x-x_\ve^{(m)})\,{\rm
d}\sigma+O(\ve) \\
=\int_{\Omega^\prime} \int_{Y\cap S} q(x,y) D w(x)\cdot y\,{\rm
d}\sigma_y{\rm d}x+o(1).
\end{multline*}
where $x_\ve^{(m)}$ is the center of the cell $Y^{(m)}_\ve$. Thus
(\ref{boundedness}) and (\ref{traces1}) are proved.

\bigskip

\noindent {\bf Step 2}(Proof of (\ref{traces}) for $w_\ve$ with
${\rm supp}(w_\ve)\subset\overline{\Omega^\prime}$). Assume now
that
\begin{equation}
w_\ve=0\ \text{in}\ \Omega\setminus\Omega^\prime\ \text{(in
particular} \ w_\ve=0\ \text{on}\ \partial\Omega^\prime\text{)} .
\label{zerotrace}
\end{equation}
Given $\delta>0$, let $\{Q_\delta^{(\alpha)}\}$ be an open cover
of $\Omega$, ${\rm diam} Q_\delta^{(\alpha)}\leq \delta$, and let
$\{\varphi_\delta^{(\alpha)}\in C^\infty(\D{R}^N)\}$ be a
partition of unity such that $$ {\rm
supp}\varphi_\delta^{(\alpha)}\subset Q_\delta^{(\alpha)}, \
0\leq\varphi_\delta^{(\alpha)}\leq 1,\
\sum_{\alpha}\varphi_\delta^{(\alpha)}=1.$$
Then we have
\begin{multline}
b_\ve w_\ve =\sum_{\alpha} \int_{S_\ve^\prime} q(\hat
x_\delta^{(\alpha)},x/\ve)(w_\ve-\bar
w_\ve)\varphi_\delta^{(\alpha)} {\rm d}\sigma \\
+ \sum_{\alpha} \int_{S_\ve^\prime} ( q(x,x/\ve)- q(\hat
x_\delta^{(\alpha)},x/\ve))(w_\ve-\bar
w_\ve)\varphi_\delta^{(\alpha)} {\rm d}\sigma=I_1+I_2,
\label{razlozhenie}
\end{multline}
where $\hat x_\delta^{(\alpha)}\in Q_\delta^{(\alpha)}$. Thanks to
the Lipschitz continuity of  $q(x,y)$ in $x$,
\begin{equation}
|I_2|\leq C\delta \sum_{\alpha} \int_{S_\ve^\prime}|w-\bar w_\ve|
\varphi_\delta^{(\alpha)}{\rm d}\sigma=
C\delta\int_{S_\ve^\prime}|w-\bar w_\ve|{\rm d}\sigma\leq
C\delta\|Dw_\ve\|_{L^2(\Omega)}. \label{vtoroi}
\end{equation}

We write the first term $I_1$ as
\begin{multline}
I_1=\sum_{\alpha} \Bigl( \int_{S_\ve^\prime}q(\hat
x^{(\alpha)}_\delta,x/\ve)w_\ve \varphi_\delta^{(\alpha)}{\rm
d}\sigma -\int_{S_\ve^\prime}q(\hat x^{(\alpha)}_\delta,x/\ve)\bar
w_\ve \varphi_\delta^{(\alpha)}{\rm d}\sigma\Bigr)=
 \sum_\alpha (\tilde I_1^{(\alpha)}+\hat I_1^{(\alpha)}).
\label{legko}
\end{multline}
Note that
\begin{equation*}
\int_{S_\ve^\prime\cap Y_\ve^{(m)}}q(\hat
x^{(\alpha)}_\delta,x/\ve) \varphi_\delta^{(\alpha)}{\rm d}\sigma=
\ve^N \Bigl( \int_{S\cap Y}q(\hat x^{(\alpha)}_\delta,y)
D\varphi_\delta^{(\alpha)}(x^{(m)}_\ve)\cdot y{\rm
d}\sigma_y+O(\ve)\Bigr)
\end{equation*}
(as above $x_\ve^{(m)}$ denotes the center of the cell
$Y_\ve^{(m)}$). Since $\bar w_\ve\to W_0(x)$ strongly in
$L^2(\Omega)$, we obtain
\begin{multline*}
\hat I_1^{(\alpha)}\to-\int_{\Omega^\prime}\Bigl(W_0(x)
 \int_{S\cap
Y}q(\hat x^{(\alpha)}_\delta,y) D\varphi_\delta^{(\alpha)}(x)\cdot
y{\rm d}\sigma_y \Bigr){\rm d}x
\\=
\int_{\Omega^\prime}\Bigl(\varphi_\delta^{(\alpha)}(x)
 \int_{S\cap
Y}q(x^{(\alpha)}_\delta,y) D W_0(x)\cdot y{\rm d}\sigma_y
\Bigr){\rm d}x ,
\end{multline*}
where we have used the fact that $W_0=0$ in
$\Omega\setminus\Omega^\prime$. Thus,
\begin{equation*}
\sum_{\alpha}\hat I_1^{(\alpha)}\to \int_{S\cap
Y}\Bigl(\sum_{\alpha}\int_{\Omega^\prime}\varphi_\delta^{(\alpha)}(x)
 q(\hat x^{(\alpha)}_\delta,y) D W_0(x)\cdot y{\rm d}x
\Bigr){\rm d}\sigma_y,
\end{equation*}
therefore
\begin{equation}
\lim_{\delta\to 0}\lim_{\ve\to 0} \sum_{\alpha}\hat
I_1^{(\alpha)}=\int_{\Omega^\prime}\int_{S\cap Y}
 q(x,y) D W_0(x)\cdot y{\rm d}\sigma_y
 {\rm d}x.
\label{25}
\end{equation}
In order to pass to the limit in $\tilde I_1^{(\alpha)}$ as
$\ve\to 0$, consider the solution $\theta$ of the problem
\begin{equation}
\begin{cases}
\Delta \theta(y)=0,\ \text{\rm in}\ Y^*;\\
\frac{\partial \theta}{\partial\nu}=q(\hat x^{(\alpha)}_\delta,y)\ \text{\rm on}\ S\cap Y;\\
\theta\ \text{\rm is}\ Y^*-\text{\rm periodic}.
\end{cases}
\label{theta}
\end{equation}
Thanks to the property (c) of $q(x,y)$ there is a unique (up to an
additive constant) solution $\theta$ of (\ref{theta}) and
$\theta\in W^{1,2}(Y^*)$. Set $\zeta_\ve(x)=\theta(x/\ve)$, then
we have $\Delta \zeta_\ve=0$ in $\Omega_\ve$ and
$\ve\frac{\partial \zeta_\ve}{\partial\nu}=q(\hat
x^{(\alpha)}_\delta,x/\ve)$ on $S_\ve^\prime$, so that
\begin{multline*}
\int_{S_\ve^\prime}q(\hat
x^{(\alpha)}_\delta,x/\ve)w_\ve\varphi_\delta^{(\alpha)}{\rm
d}\sigma
=\ve\int_{S_\ve^\prime}w_\ve\varphi_\delta^{(\alpha)}\frac{\partial
\zeta_\ve}{\partial\nu}{\rm d}\sigma\\
=\ve\int_{\Omega_\ve\cap\Omega^\prime}D(w_\ve\varphi_\delta^{(\alpha)})\cdot
D \zeta_\ve\,{\rm d}x=
\int_{\Omega_\ve\cap\Omega^\prime}D(w_\ve\varphi_\delta^{(\alpha)})\cdot
(D\theta)(x/\ve)\,{\rm d}x.
\end{multline*}
(we have taken into account here that $w_\ve=0$ on $\partial
\Omega^\prime$). One easily checks that
$D(w_\ve\varphi^{(\alpha)}_\delta)(x)\to
D(W_0\varphi^{(\alpha)}_\delta)(x)+\varphi^{(\alpha)}_\delta D_y
W_1(x,y)$ two-scale, therefore
\begin{multline*}
\tilde I_1^{(\alpha)} \to \int_{\Omega^\prime}\Bigl(\int_{Y^*}
(D(W_0\varphi_\delta^{(\alpha)})+\varphi_\delta^{(\alpha)}D_yW_1(x,y))\cdot
(D\theta)(y) \,{\rm d}y\Bigr)
\,{\rm d}x\\
=\int_{\Omega^\prime}\varphi_\delta^{(\alpha)}\Bigl(\int_{S\cap
Y}W_1(x,y)q(\hat x^{(\alpha)}_\delta,y) \,{\rm d}\sigma_y\Bigr)
\,{\rm d}x,
\end{multline*}
where we have used (\ref{theta}). Thus, taking into account the
Lipschitz continuity of $q(x,y)$ in $x$, we get, passing to the
limit as $\delta\to 0$,
\begin{multline}
\lim_{\delta\to 0}\lim_{\ve\to 0} \sum_{\alpha} \tilde
I_1^{(\alpha)}= \sum_{\alpha}
\int_{\Omega^\prime}\varphi_\delta^{(\alpha)}\Bigl(\int_{S\cap
Y}W_1(x,y)q(x,y) \,{\rm d}\sigma_y\Bigr) \,{\rm d}x
\\
=\int_{\Omega^\prime}\int_{S\cap
Y}W_1(x,y)q(x,y) \,{\rm d}\sigma_y {\rm d}x,
\label{rezultat}
\end{multline}
and we finally obtain by (\ref{razlozhenie}) -
(\ref{25}), (\ref{rezultat}),
\begin{equation}
\int_{S_\ve^\prime}q(x,x/\ve)(w_\ve-\bar w_\ve){\rm d}\sigma\to
\int_{\Omega^\prime} \int_{Y\cap S}q(x,y)(DW_0\cdot
y+W_1(x,y)){\rm d}\sigma_y {\rm d}x. \label{traces11}
\end{equation}

\bigskip
\noindent {\bf Step 3}(general case). Let $(w_\ve)$ be now an
arbitrary sequence such that $w_\ve\to W_0$ weakly in
$W^{1,2}(\Omega)$, and $Dw_\ve\to DW_0(x)+D_yW_1(x,y)$ two-scale.
Write $w_\ve=(w_\ve-(W_0 +w_\ve^{(1)}))+ w^{(1)}_\ve+W_0$, where
$w_\ve^{(1)}$ is the unique solution of the problem
\begin{equation*}
\begin{cases}
\Delta w_\ve^{(1)}= 0\ \text{\rm in}\ \Omega^\prime\\
w_\ve^{(1)}=w_\ve-W_0\ \text{\rm on}\ \partial\Omega^\prime,
\end{cases}
\end{equation*}
extended in $\Omega\setminus\Omega^\prime$ by setting
$w_\ve^{(1)}=w_\ve-W_0$.  Since $w_\ve-W_0\to 0$ weakly in
$H^{1/2}(\partial\Omega^\prime)$, we have
\begin{equation}
w_\ve^{(1)}\to 0 \ \text{\rm strongly in}\ W^{1,2}(K) \ \text{for
any compact}\ K\subset \Omega^\prime, \label{strong}
\end{equation}
by standard elliptic estimates. This implies, in particular, that
$w_\ve^{(1)}\to 0$, $Dw_\ve^{(1)}\to 0$ two-scale. Moreover in
view of (\ref{rescaledpoincare}), for any compact subset $K$ of
$\Omega^\prime$,
\begin{multline}
|b_\ve w_\ve^{(1)}| \leq C\sum_{m:Y^{(m)}_\ve\cap K\not=\emptyset}
\int_{Y^{(m)}_\ve}|w_\ve^{(1)}-\bar w_\ve^{(1)}|{\rm d}\sigma+
C\sum_{m:Y^{(m)}_\ve\cap
K=\emptyset}\int_{Y^{(m)}_\ve\cap\Omega^\prime} |w_\ve^{(1)}-\bar
w_\ve^{(1)}|{\rm d}\sigma
\\
\leq C \Bigl(\int_{K_\delta}|D w_\ve^{(1)}|^2{\rm d}x\Bigr)^{1/2}+
C|\Omega^\prime_\delta\setminus K|^{1/2}
\Bigl(\int_{\Omega}|Dw_\ve^{(1)}|^2{\rm d}x\Bigr)^{1/2},
\label{30}
\end{multline}
when $\ve\leq \delta/N$ , where $C$ is independent of $\ve$ and
$\delta$, $K_\delta,\Omega^\prime_\delta$ are the
$\delta$-neighborhoods of $K$ and $\Omega^\prime$, respectively,
and $\delta>0$ is arbitrary. (The summation in (\ref{30}) is taken
over $m$ such that $Y_\ve^{(m)}\cap\Omega^\prime\not=\emptyset$.)
It follows from (\ref{strong}), (\ref{30}) that $b_\ve
w_\ve^{(1)}\to 0$ as $\ve\to 0$, while, according to the first and
second steps,
$$
b_\ve W_0\to \int_{\Omega^\prime} \int_{Y\cap S}q(x,y)DW_0\cdot
y{\rm d}\sigma_y {\rm d}x,
$$
and
$$
b_\ve (w_\ve-(W_0+ w_\ve^{(1)}))\to \int_{\Omega^\prime}
\int_{Y\cap S}q(x,y)W_1(x,y){\rm d}\sigma_y {\rm d}x.
$$
Thus (\ref{traces11}) is proved for any sequence $(w_\ve)$ such
that (\ref{convergwve}) holds.

\bigskip
\noindent {\bf Final step}. Set $\Omega^\prime=\{x\in\Omega; {\rm
dist}(x,\partial\Omega)>\delta\}$, where $\delta>0$. By using
(\ref{rescaledpoincare}) we have,
\begin{equation}
\int_{S_\ve\setminus S_\ve^\prime}|w_\ve-\bar w_\ve|{\rm
d}\sigma\leq C\frac{\delta^{1/2}}{\ve^{1/2}}
\Bigl(\int_{S_\ve\setminus S_\ve^\prime}|w_\ve-\bar w_\ve|^2{\rm
d}\sigma\Bigr)^{1/2}\leq
 C\delta^{1/2}
\|w_\ve\|_{W^{1,2}(\Omega)}, \label{vozlegranitsy}
\end{equation}
for sufficiently small $\ve$, where $C$ is independent of $\delta$
and $\ve$.  Therefore (\ref{vozlegranitsy}) combined with
(\ref{traces11}) yield (\ref{traces}) for any sequence $(w_\ve)$
such that (\ref{convergwve}) holds. \hfill$\square$

\bigskip

\noindent
{\bf 2}({\it Proof of} (\ref{middleterm})). We
approximate $U_0$ by functions $u^{(1)}_\delta\in
C^1(\overline{\Omega})$ ($\delta>0$) in the strong topology of
$L^2(\Omega)$, $\|U_0-u^{(1)}_\delta\|_{L^2(\Omega)}\leq \delta$.
By virtue of Lemma \ref{techlem1}, the strong-$L^2$ convergence of
$u_\ve$ to $U_0$ and Lemma \ref{techlem2} we then have
\begin{multline}
\limsup_{\ve\to 0} \biggl|
\int_{S_\ve}g(u_\ve,x/\ve)(u_\ve-v_\ve){\rm d}\sigma-
\int_{S_\ve}g(\bar
u_\delta^{(1)},x/\ve)(u_\ve-v_\ve-\bar u_\ve+\bar v_\ve){\rm d}\sigma \biggr.\\
\biggl.- \int_{S_\ve}g^\prime_u(\bar u_\delta^{(1)},x/\ve) (\bar
u_\delta^{(1)}-\bar v_\ve)(u_\ve-\bar u_\ve){\rm
d}\sigma\biggr|\leq C\delta. \label{pervoedeistvie}
\end{multline}
On the other hand, the regularity of $g(u,y)$ in $u$ (conditions
(\ref{4}), (\ref{5}), (\ref{6})) implies the pointwise bounds
$$|g(\bar
u_\delta^{(1)},x/\ve)-g(u_\delta^{(1)},x/\ve)| \leq C\ve\
\text{on}\ S_\ve ,
$$
$$
|g^\prime_u(\bar u_\delta^{(1)},x/\ve) (\bar u_\delta^{(1)}-\bar
v_\ve)- g^\prime_u(u_\delta^{(1)},x/\ve) ( u_\delta^{(1)}-V_0)|
\leq C\ve \ \text{on}\ S_\ve
$$ (recall that $v_\ve=V_0(x)+\ve V_1(x,x/\ve)$, and
$V_0$, $V_1$ are smooth functions), which, by using
(\ref{rescaledpoincare}), lead to
\begin{multline}
\limsup_{\ve\to 0} \biggl| \int_{S_\ve} (g(\bar
u_\delta^{(1)},x/\ve)-g(u_\delta^{(1)},x/\ve))
(u_\ve-v_\ve-\bar u_\ve+\bar v_\ve){\rm d}\sigma \biggr.\\
\biggl.+ \int_{S_\ve}( g^\prime_u(\bar u_\delta^{(1)},x/\ve) (\bar
u_\delta^{(1)}-\bar v_\ve)- g^\prime_u(u_\delta^{(1)},x/\ve) (
u_\delta^{(1)}-V_0))(u_\ve-\bar u_\ve){\rm d}\sigma\biggr|=0.
\label{vtoroedeistvie}
\end{multline}
Now, applying Proposition \ref{convergtraces} first with
$q(x,y)=g(u^{(1)}_\delta(x),y)$, $w_\ve=u_\ve-v_\ve$, then with
$q(x,y)=g(u^{(1)}_\delta(x),y)u^{(1)}_\delta(x)$, $w_\ve=u_\ve$,
and finally with $q(x,y)=g(u^{(1)}_\delta(x),y)V_0(x)$,
$w_\ve=u_\ve$, we get
\begin{multline}
\int_{S_\ve} \bigl(g(u_\delta^{(1)},x/\ve)(u_\ve-v_\ve-\bar
u_\ve+\bar v_\ve) + g^\prime_u(u_\delta^{(1)},x/\ve) (
u_\delta^{(1)}-V_0)(u_\ve-\bar u_\ve)\bigr){\rm d}\sigma\\
\to \int_\Omega\int_{S\cap Y} g(u_\delta^{(1)},y)(D(U_0-V_0)\cdot
y+
U_1(x,y)-V_1(x,y)){\rm d}\sigma_y{\rm d}x\\
+\int_\Omega\int_{S\cap Y}
g^\prime_u(u_\delta^{(1)},y)(u_\delta^{(1)}-V_0)(DU_0\cdot y+
U_1(x,y)){\rm d}\sigma_y{\rm d}x. \label{tretedeistvie}
\end{multline}
Assuming  $\delta\to 0$ in (\ref{pervoedeistvie}),
(\ref{vtoroedeistvie}), (\ref{tretedeistvie}) yields
(\ref{middleterm}). \hfill$\square$


\section{Homogenization of the parabolic problem (\ref{eq2})}
\label{section6}

In terms of the operators ${\C A}_\ve$ and ${\C G}_\ve$ problem
(\ref{eq2}) is as follows
\begin{equation}
\begin{cases}
\partial_t u_\ve(t)+{\C A}_\ve(u_\ve(t))-{\C
G}_\ve(u_\ve(t))=f(t),\  t>0 \\
u_\ve(0)=\tilde u.
\end{cases}
\label{parabolic}
\end{equation}
We study the asymptotic behavior of solutions $u^\ve$ of
(\ref{parabolic}) as $\ve\to 0$ adapting the notion of two-scale
convergence to functions depending on the time variable $t$ which
is treated as a parameter. Namely, following \cite{CP} we say that
\begin{equation}
\begin{array}{l}
 \text{the sequence}\ v_\ve=v_\ve(x,t)\ \text{which is
bounded in}\ L^2(\Omega\times [0,T]) \\ \text{two-scale}\
\text{converges to}\ V_0(x,y,t)\ \text{if} \\
\displaystyle \int_{0}^T\int_\Omega v_\ve\phi(x,x/\ve,t)\,{\rm
d}x{\rm d} t\to \int_{0}^T\int_{Y}\int_\Omega V_0\phi(x,y,t)
\,{\rm d}x{\rm d} y{\rm d} t,\\
\text{for any}\ Y-\text{periodic in}\ y\  \text{function}\
\phi(x,y,t)\in C^\infty(\Omega\times Y\times [0,T]).
\end{array}
\label{parametertwoscale}
\end{equation}
The basic properties of the convergence (\ref{parametertwoscale})
are similar to that of the standard two-scale convergence. Namely,
any bounded in $L^2(\Omega\times [0,T])$ sequence has a
subsequence converging in the sense of (\ref{parametertwoscale});
if $\|v_\ve\|_{L^2(0,T;W^{1,2}(\Omega))}\leq C$ then, up to
extracting a subsequence, $v_\ve$ and $Dv_\ve$ converge in the
sense of (\ref{parametertwoscale}) to $V_0$ and
$DV_0(x,t)+D_yV_1(x,y,t)$ correspondingly, where $V_0\in
L^2(0,T;W^{1,2}(\Omega))$, $V_1\in
L^2([0,T]\times\Omega;W^{1,2}_{per}(Y))$. Note, however, that
(\ref{parametertwoscale}) does not imply, in general, that
$v_\ve(\, \cdot\,,t)$ converges in two-scale sense for a.e. $t\in
[0,T]$, but rather
$$
\int_\alpha^\beta v_\ve {\rm d} t\to \int_\alpha^\beta V_0 {\rm d}
t\ \text{two scale for all}\ 0\leq\alpha<\beta\leq T.
$$


\bigskip

\noindent {\bf 1} ({\it Well-posedness of problem}
(\ref{parabolic})). Given $T>0$, let us show that problem
(\ref{parabolic}) has a unique solution on the time interval
$[0,T]$. To this end we first note that the operator ${\C
A}_\ve(u)-{\C G}_\ve(u)+\tilde\lambda u$ becomes monotone if one
chooses a suitable $\tilde\lambda>0$ (depending on $\ve$). Indeed,
 by using (\ref{5}) we get
\begin{multline}
\langle{\C G}_\ve(u)-{\C G}_\ve(v),u-v\rangle_{\ve}\leq C
\int_{S_\ve}| u-v|^2{\rm d}\sigma\\
\leq \kappa/2\|D(u-v)\|^2_{L^{2}(\Omega_\ve)}+
\Gamma_\ve\|u-v\|^2_{L^{2}(\Omega_\ve)},\ \forall u,v\in
W^{1,2}(\Omega_\ve). \label{podporka}
\end{multline}
where $\kappa$ is the constant appearing in (\ref{monoton}), and
$\Gamma_\ve$ is independent of $u_\ve$ and $v_\ve$ (the last
inequality in (\ref{podporka}) is due the compactness of the trace
operator $T_\ve:W^{1,2}(\Omega_\ve)\to L^2(S_\ve)$, $T_\ve
w=\text{trace of}\ w\ \text{on}\ S_\ve$ ). Then, setting $\tilde
\lambda=\Gamma_\ve+1$, by (\ref{monoton}) and (\ref{podporka}) one
easily verifies that
\begin{equation}
\text{the operator}\ u\mapsto
{\C A}_\ve(u)-{\C
G}_\ve(u)+\tilde\lambda u\ \text{is monotone}
\label{shiftedmonotonicity}
\end{equation}

By changing the unknown $v_\ve=e^{-\tilde\lambda t} u_\ve$  problem (\ref{parabolic})
is reduced to the evolution problem for
the equation $\partial_t v_\ve(t)+\tilde{\C A}_\ve(v_\ve(t),t)-\tilde{\C
G}_\ve(v_\ve(t),t)+\tilde\lambda v_\ve=e^{-\tilde\lambda t} f(t),\  t>0$ with
the initial condition $v_\ve(0)=\tilde u$, where
$\tilde{\C A}_\ve: v\mapsto e^{-\tilde\lambda t}{\C A}_\ve(e^{\tilde\lambda t}v)$  and
$\tilde{\C G}_\ve: v\mapsto e^{-\tilde\lambda t}{\C G}_\ve(e^{\tilde\lambda t}v)$.
By the standard theory of parabolic
problems for monotone operators
(see, e.g. \cite{S}) it follows from  (\ref{shiftedmonotonicity}), (\ref{monoton})
and (\ref{podporka})
that the latter problem has a unique solution on $[0,T]$ as far as
$f\in L^2([0,T];X_\ve^*)$ and $\tilde u\in L^2(\Omega)$.

\bigskip

\noindent {\bf 2} ({\it Uniform a-priori bounds}). Let us show
that for any $T>0$ the solution $u_\ve$ of (\ref{parabolic})
satisfies the following bounds for sufficiently small $\ve$,
\begin{equation}
\|\partial_t u_\ve\|^2_{L^2(0,T; X^*_\ve)},\|u_\ve\|^2_{L^2(0,T;
X_\ve)}\leq C ( \langle \tilde u,\tilde
u\rangle_\ve+\|f\|^2_{L^2(0,T; X^*_\ve)}+1),
\label{uniformaprioriparabolic1}
\end{equation}
with a constant $C$ independent of $\ve$. Let $\ve_0$, $\lambda_0$
be as in Theorem \ref{aprioriestimate}. From (\ref{parabolic}) we
have, for $\ve\leq\ve_0$
\begin{multline}
\langle u_\ve(t),u_\ve(t) \rangle_{\ve}+ 2\int_{0}^{t} \langle{\C
A}_\ve(u_\ve(\tau))+ {\C G}_\ve(u_\ve(\tau))+ \lambda_0
u_\ve(\tau),u_\ve(\tau)\rangle_{\ve}{\rm d}\tau\\
=\langle \tilde u,\tilde u\rangle_{\ve}+2\int_{0}^{t} \langle
f(\tau)+\lambda_0u_\ve(\tau),u_\ve(\tau) \rangle_{\ve}{\rm d}\tau.
\label{ravenstvo}
\end{multline}
Then (\ref{ravenstvo}) combined with (\ref{coercivity}) yields
\begin{multline}
\langle u_\ve(T^\prime),u_\ve(T^\prime)\rangle_{\ve}+
2\kappa_1\|u_\ve\|^2_{L^2(0,T^\prime; X_\ve)}\leq \langle \tilde
u,\tilde u\rangle_\ve +\|f\|_{L^2(0,T^\prime; X^*_\ve)}
\|u_\ve\|_{L^2(0,T^\prime;
X_\ve)}\\
+2T^\prime\kappa_2 +2\lambda_0 \int_{0}^{T^\prime} \langle
u_\ve(t),u_\ve(t)\rangle_{\ve}{\rm d}t, \ \forall \ 0\leq
T^\prime\leq T. \label{anotherbound}
\end{multline}
Therefore
\begin{equation}
\langle u_\ve(T^\prime),u_\ve(T^\prime) \rangle_{\ve}\leq
e^{2\lambda_0 T^\prime}(\langle \tilde u,\tilde u
\rangle_\ve
+\frac{1}{\kappa_1} \|f\|^2_{L^2(0,T^\prime; X^*_\ve)}
+2T^\prime\kappa_2), \label{anotherbound1}
\end{equation}
combined with  (\ref{anotherbound}) this implies the second bound
in (\ref{uniformaprioriparabolic1}); while $\|\partial_t
u_\ve\|_{L^2(0,T; X^*_\ve)}\leq \|{\C A}_\ve(u_\ve)\|_{L^2(0,T;
X^*_\ve)}+\|{\C G}_\ve(u_\ve)\|_{L^2(0,T;
X^*_\ve)}+\|f\|_{L^2(0,T; X^*_\ve)}$ and thus the first bound in
(\ref{uniformaprioriparabolic1}) is a consequence of the second
one and (\ref{boundforG}).

\bigskip

\noindent {\bf 3} ({\it Homogenization of problem}
(\ref{parabolic})). Let $u_\ve$ be continued in $x$ variable onto
$\Omega$ by using the extension operator $P_\ve$, then the
resulting function, still denoted $u_\ve$, satisfies
\begin{equation}
\|u_\ve(t)\|_{L^2(\Omega)}\leq C\ \text{for all}\ t\in[0,T], \
\text{and}\ \|u_\ve\|_{L^2(0,T;W^{1,2}(\Omega))}\leq C,
\label{uniformaprioriparabolicextended1}
\end{equation}
with a constant $C$ independent of $\ve$. This implies that, up to
extracting a subsequence,
\begin{equation}
u_\ve\to U_0(x,t)\ \text{two-scale (in the sense of
(\ref{parametertwoscale})) and weakly in}\ L^2(0,T;
W^{1,2}(\Omega)), \label{2scale1}
\end{equation}
\begin{equation}
D_xu_\ve\to D_xU_0(x,t)+D_yU_{1}(x,y,t)\ \text{two-scale (in the
sense of (\ref{parametertwoscale}))}, \label{2scale2}
\end{equation}
where $U_0\in L^2(0,T;W^{1,2}(\Omega))$, $U_1\in
L^2(0,T;L^2(\Omega;W^{1,2}_{per}(Y)))$. Besides, if we set $\hat
u_\ve=u_\ve$ when $x\in\Omega_\ve$ and $\hat u_\ve=0$ when $x\in
\Omega\setminus \Omega_\ve$, then (\ref{2scale1}) yields that
$\hat u_\ve\to |Y^*|U_0(x,t)$ weakly in $L^2(0,T;L^{2}(\Omega))$.

Let $X=W^{1,2}(\Omega)$ an let $X^*$ be its dual with respect to
the duality pairing
$$
\langle u,v\rangle=|Y^*|\int_\Omega uv {\rm d}x.
$$
Show that $U_0\in W^{1,2}(0,T;X^*)$, and $\hat u_\ve(t)\to
|Y^*|U_0(t)$ weakly in $L^2(\Omega)$ for all $0\leq t\leq T$. From
(\ref{2scale1}) we have, for any $\phi\in X$ and $\varphi\in
C^\infty_0([0,T])$,
\begin{equation}
 \int_0^T \langle\partial_t u_\ve, \phi\rangle_\ve\varphi(t){\rm d} t=
-\int_0^T \langle u_\ve, \phi\rangle_\ve\varphi^\prime (t){\rm d}
t\to-\int_0^T \langle U_0, \phi\rangle\varphi^\prime (t){\rm d} t.
\label{smooth1}
\end{equation}
On the other hand, by using (\ref{uniformaprioriparabolic1}), we
get
\begin{equation} \biggl|\int_0^T \langle\partial_t u_\ve,
\phi\rangle_\ve\varphi(t){\rm d} t\biggr|^2\leq C\int_0^T
\|\phi\|^2_{X_\ve}|\varphi(t)|^2{\rm d} t\leq
C\|\varphi\phi\|^2_{L^2(0,T;X)} . \label{smooth2}
\end{equation}
Then (\ref{smooth1}), (\ref{smooth2}) show that $U_0\in
W^{1,2}(0,T;X^*)$. According to (\ref{anotherbound1}), the norms
$\|\hat u_\ve(t)\|_{L^2(\Omega)}$ are uniformly in
$0<\ve\leq\ve_0$ and $t\in[0,T]$ bounded. Thus, to prove that
$\hat u_\ve(t)\to |Y^*|U_0(t)$ weakly in $L^2(\Omega)$ for every
$t\in [0,T]$ it suffices to show that
\begin{equation}
\langle u_\ve(t), \phi\rangle_\ve\to \langle U_0(t),\phi\rangle \
\text{for any} \ \phi\in X. \label{weakpointwise}
\end{equation}
By the first bound in (\ref{uniformaprioriparabolic1}) we have
$|\langle u_\ve(t)-u_\ve(t^\prime), \phi\rangle_\ve|\leq
C|t-t^\prime|^{1/2}\|\phi\|_X$, on the other hand
(\ref{weakpointwise}) holds in the sense of weak star convergence
in $L^\infty(0,T)$ since $\hat u_\ve\to |Y^*|U_0(x,t)$ weakly in
$L^2(0,T;L^{2}(\Omega))$. Thus (\ref{weakpointwise}) holds for any
$t\in [0,T]$, so that $\forall t\in [0,T]$ $\hat u_\ve(t)\to
|Y^*|U_0(t)$ weakly in $L^2(\Omega)$, in particular,
\begin{multline}
\liminf_{\ve\to 0}\langle u_\ve(T),u_\ve(T)\rangle_\ve
=\liminf_{\ve\to
0}\int_{\Omega_\ve}((u_\ve(T)-U_0(T))^2-U_0^2(T))\, {\rm d}x\\
+2 \lim_{\ve\to 0}\int_{\Omega}\hat u_\ve(T)U_0(T)\, {\rm d}x
=\liminf_{\ve\to 0}\int_{\Omega_\ve}(u_\ve(T)-U_0(T))^2\, {\rm d}x
+\langle U_0(T),U_0(T)\rangle
\\
\geq \langle U_0(T),U_0(T)\rangle,
\label{lowersemicont1}
\end{multline}
and, clearly,
\begin{multline}
\langle u_\ve(T),v_\ve\rangle_\ve \to \langle U_0(T),V_0\rangle, \
\text{for any sequence}\ v_\ve\to V_0 \ \text{strongly in}\
L^2(\Omega). \label{lowersemicont2}
\end{multline}

\begin{lem}
\label{lemstrongL2conver}
 If $(u_\ve)$ is such a (sub)sequence
of solutions of (\ref{parabolic}) that (\ref{2scale1}) holds, then
\begin{equation}
\|u_\ve-U_0\|_{L^2(\Omega\times [0,T])}\to 0 \ \text{as}\ \ve\to
0. \label{strongL2conver}
\end{equation}
\end{lem}

\begin{proof} By (\ref{2scale1}) it suffices to establish the
(relative) compactness of $(u_\ve)$ in $L^2(\Omega\times [0,T])$.
This is achieved by constructing a sequence of compacts $K_k$
($k=1,2,\dots$) in $L^2(\Omega\times [0,T])$ such that
$\lim_{k\to\infty}\limsup_{\ve\to 0} {\rm dist}_{L^2(\Omega\times
[0,T])}(u_\ve, K_k)=0$.

Let $0=\omega_\ve^{(0)}<\omega_\ve^{(1)}\leq\dots\leq
\omega_\ve^{(j)}\leq\dots$ be the spectrum of the Neumann
eigenvalue problem
\[
\begin{cases}
-\Delta \phi=\omega\phi\ \text{in}\ \Omega_\ve\\
\frac{\partial \phi}{\partial \nu}=0\ \text{on}\
\partial\Omega_\ve.
\end{cases}
\]
The eigenfunctions $\phi_\ve^{(j)}$ can be chosen to form an
orthogonal basis of $L^2(\Omega_\ve)$, then
\[
u_\ve(t)=\sum_{j=0}^\infty f_\ve^{(j)}(t)P_\ve\phi_\ve^{(j)},\
\text{where}\ f_\ve^{(j)}(t)=\langle
u_\ve(t),\phi_\ve^{(j)}\rangle_\ve.
\]
Moreover $\phi_\ve^{(j)}/(\omega_\ve^{(j)}+1)^{1/2}$ form an
orthonormal basis in $X_\ve(=W^{1,2}(\Omega_\ve)$, hence
\begin{equation}
\sum_{j=0}^\infty
(1+\omega_\ve^{(j)})\int_0^T|f_\ve^{(j)}(t)|^2{\rm d} t
=\|u_\ve\|^2_{L^2(0,T;X_\ve)}\leq \|u_\ve\|^2_{L^2(0,T;X)}\leq C.
\label{fourier}
\end{equation}

 It is well known that $\omega_\ve^{(k)}\to \omega^{(k)}$ as
$\ve\to 0$, where $0=\omega^{(0)}<\omega^{(1)}\leq\dots\leq
\omega^{(j)}\leq\dots$ is the discrete spectrum of a homogenized
problem. By the first bound in (\ref{uniformaprioriparabolic1}) we
have $|f_\ve^{(j)}(t)-f_\ve^{(j)}(t^\prime)|\leq
C|t-t^\prime|^{1/2}\|\phi_\ve^{(j)}\|_{X_\ve}=
C|t-t^\prime|^{1/2}(1+\omega_\ve^{(j)})^{1/2}$ for all $t,t^\prime
\in [0,T]$. It follows that, for every $k$ fixed, the sequence
$(u_\ve^{(k)}:=\sum_{j=0}^kf_\ve^{(j)}(t)P_\ve\phi_\ve^{(j)})$ is
in a bounded closed  subset $K_k$ of  $C^{1/2}([0,T];X)$. Clearly
$K_k$ is a compact set in $L^2(\Omega\times [0,T])$. On the other
hand, due to the properties of the extension operator $P_\ve$,
$$
\|u_\ve-u_\ve^{(k)}\|^2_{L^2(\Omega\times [0,T])}\leq C\int_0^T
\|u_\ve-u_\ve^{(k)}\|^2_{L^2(\Omega_\ve)}{\rm
d}t=C\sum_{j=k+1}^\infty\int_0^T|f_\ve^{(j)}(t)|^2{\rm d} t,
$$
therefore, in view of (\ref{fourier}), $\limsup_{\ve\to 0} {\rm
dist}_{L^2(\Omega\times [0,T])}(u_\ve, K_k)\leq \limsup_{\ve\to 0}
\|u_\ve-u_\ve^{(k)}\|_{L^2(\Omega\times [0,T])}\leq
 C/\omega^{(k+1)}\to
0$ as $k\to\infty$.
\end{proof}

Now, set $V_0(x,t)\in C^{\infty}(\overline{\Omega}\times [0,T])$,
$V_1(x,y,t)\in C^{\infty}(\overline{\Omega}\times
\overline{Y}\times [0,T])$ with $V_1(x,y,t)$ being $Y$-periodic in
$y$, set $v_\ve =V_0(x,t)+\ve V_{1}(x,x/\ve,t)$, and using the
test function $w_\ve =u_\ve-v_\ve$ in (\ref{parabolic}) we obtain
\begin{multline}
\frac{1}{2}\langle u_\ve(T),u_\ve(T)\rangle_{\ve}- \frac{1}{2}
\langle\tilde u,\tilde u\rangle_{\ve}-\langle
u_\ve(T),v_\ve(T)\rangle_{\ve}+\langle\tilde
u,v_\ve(0)\rangle_{\ve}
\\
+\int_{0}^{T} \langle u_\ve(t),\partial_t
v_\ve(t)\rangle_{\ve}{\rm d}t +\int_{0}^{T} \langle{\C
A}_\ve(u_\ve(t)),w_\ve(t)\rangle_{\ve}{\rm d}t-
\int_{0}^{T}\langle{\C G}_\ve(u_\ve(t)),w_\ve(t)\rangle_{\ve}{\rm d}t \\
=\int_{0}^{T} \langle f(t),w_\ve(t)\rangle_{\ve}{\rm d}t.
\label{ravenstvoparab}
\end{multline}
By using (\ref{2scale1}) and (\ref{lowersemicont1}),
(\ref{lowersemicont2}), we can take $\liminf_{\ve\to 0}$ for
various terms in (\ref{ravenstvoparab}) to get
\begin{multline}
\liminf_{\ve\to 0}\biggl( \frac{1}{2}\langle
u_\ve(T),u_\ve(T)\rangle_{\ve}- \frac{1}{2} \langle\tilde u,\tilde
u\rangle_{\ve}-\langle
u_\ve(T),v_\ve(T)\rangle_{\ve}+\langle\tilde
u,v_\ve(0)\rangle_{\ve}\biggr.\\
\biggl.+\int_{0}^{T} (\langle u_\ve(t),\partial_t
v_\ve(t)\rangle_{\ve}-\langle f(t),w_\ve(t)\rangle_{\ve}){\rm
d}t\biggr)
\\
\geq \frac{1}{2} \langle U_0(T),U_0(T)\rangle - \frac{1}{2}
\langle\tilde u,\tilde u\rangle -\langle U_0(T),V_0(T)\rangle
+\langle\tilde u,V_0(0)\rangle\\
+\int_{0}^{T} (\langle U_0(t),\partial_t V_0(t)\rangle- \langle
f(t),U_0(t)-V_0(t)\rangle){\rm d}t. \label{firstsemifinal}
\end{multline}
By (\ref{2scale2}) we also have
\begin{multline}
\lim_{\ve\to 0} \int_{0}^{T} \langle{\C
A}_\ve(v_\ve(t)),w_\ve(t)\rangle_{\ve}{\rm d}t \\
= \int_{0}^{T}\int_\Omega \int_{Y^*} a(D_xV_0+D_yV_1,y)\cdot
(D_xU_0+D_yU_1-D_xV_0-D_yV_1){\rm d}y{\rm d}x{\rm d} t.
\label{secondsemifinal}
\end{multline}
Let us show that
\begin{equation}
\int_0^T\langle \C{G}_\ve(u_\ve),u_\ve-v_\ve \rangle_\ve{\rm d}
t\to \int_0^T M(U_0,U_1,V_0,V_1){\rm d} t \ \text{as}\ \ve\to 0,
\label{thusconvergenceint}
\end{equation}
where $M(U_0,U_1,V_0,V_1)$ is given by
(\ref{middletermrepresentation}) (or, equivalently, by the r.h.s.
of (\ref{middleterm})). The proof of (\ref{thusconvergenceint})
follows closely the arguments in the end of Sec. \ref{section5}
(proof of (\ref{middleterm})). In place of Proposition
\ref{convergtraces} we make use now of


\begin{prop}
\label{convergtraces1} Assume that $q(t,x,y)\in C([0,T]\times
\Omega;L^\infty(S))$ satisfies, {\rm (a)}
$|q(t,x,y)-q(t^\prime,x^{\prime},y)|\leq
C(|x-x^\prime|+|t-t^\prime|)$ with $C>0$ independent of
$x,x^\prime \in\Omega$, $t,t^\prime \in[0,T]$
 and $y\in S$;
{\rm (b)} $q(t,x,y)$ is $Y$-periodic in $y\in S$;

\noindent {\rm (c)} $\displaystyle\int_{Y\cap S}q(t,x,y){\rm
d}\sigma_y=0$ for all $x\in\Omega$, $t\in[0,T]$.

 Then, given a sequence $w_\ve\in
L^2(0,T;W^{1,2}(\Omega))$ such that $w_\ve\to W_0$,
$D_xw_\ve(x,t)\to D_xW_0(x,t)+D_yW_1(x,y,t)$ two scale (in the
sense of (\ref{parametertwoscale}))  as $\ve\to 0$, we have
\begin{equation}
\int_0^T\int_{S_\ve}q(t,x,x/\ve)(w_\ve-\bar w_\ve){\rm
d}\sigma{\rm d} t\to \int_0^T\int_\Omega \int_{Y\cap
S}q(t,x,y)(D_xW_0\cdot y+W_1){\rm d}\sigma_y {\rm d}x{\rm d} t.
\label{traces1}
\end{equation}
\end{prop}

\begin{proof} Set
$0=t_0^{(n)}<\dots<t_j^{(n)}=Tj/n<\dots<t_n^{(n)=T}$,
$\Delta_j^{(n)}=(t_{j-1}^{(n)},t_{j}^{(n)})$, then by using
(\ref{rescaledpoincare}) and the Lipschitz continuity of
$q(t,x,y)$ in $t$ we obtain
\begin{multline}
\int_0^T\int_{S_\ve}q(t,x,x/\ve)(w_\ve-\bar w_\ve){\rm
d}\sigma{\rm d} t=\sum_{j=1}^n\int_{\Delta_j^{(n)}}
\int_{S_\ve}q(t,x,x/\ve)(w_\ve-\bar w_\ve){\rm d}\sigma{\rm d}
t
\\
= \sum_{j=1}^n\int_{S_\ve}
q(t_j^{(n)},x,x/\ve)\int_{\Delta_j^{(n)}}(w_\ve-\bar w_\ve){\rm d}
t{\rm d}\sigma +r^{(n)}_\ve, \label{riemann}
\end{multline}
with
\begin{equation}
|r^{(n)}_\ve|\leq \frac{C}{n} \int_0^T\int_{S_\ve} |w_\ve-\bar
w_\ve|{\rm d}\sigma{\rm d} t \leq
\frac{C}{n}\int_0^T\|w_\ve\|_{W^{1,2}(\Omega)}{\rm d} t.
\label{riemannremainder}
\end{equation}
Setting $W_\ve=\int_{\Delta_j^{(n)}}w_\ve{\rm d} t$ and applying
Proposition \ref{convergtraces}, we get
\begin{equation}
\label{riemannfinal}
 \lim_{\ve\to 0}\int_{S_\ve} q(t_j^{(n)},x,x/\ve)
(W_\ve-\bar W_\ve){\rm d} t{\rm d}\sigma=
\int_{\Delta_j^{(n)}}\int_\Omega \int_{Y\cap
S}q(t_j^{(n)},x,y)(D_xW_0\cdot y+W_1){\rm d}\sigma_y {\rm d}x{\rm
d} t.
\end{equation}
If we pass to the limit (along a subsequence) as $\ve\to 0$ in
(\ref{riemann})  and send $n$ to $\infty$ in the resulting relation, then by
(\ref{riemannremainder}) and (\ref{riemannfinal}) we obtain
(\ref{traces1}).
\end{proof}

\noindent {\it Proof of (\ref{thusconvergenceint}) (continued).}
By virtue of Lemma \ref{techlem1} and
(\ref{uniformaprioriparabolicextended1}) we have
\begin{multline*}
\int_0^T\langle \C{G}_\ve(u_\ve),u_\ve-v_\ve \rangle_\ve{\rm d} t=
\int_0^T \int_{S_\ve}g(\bar u_\ve,x/\ve)(u_\ve-v_\ve-\bar
u_\ve+\bar v_\ve){\rm d}\sigma{\rm d} t
\\
+ \int_0^T\int_{S_\ve}g^\prime_u(\bar u_\ve,x/\ve) (\bar
u_\ve-\bar v_\ve) (u_\ve-\bar u_\ve){\rm d}\sigma{\rm d} t
+O(\ve^{2/(N+2)}),
\end{multline*}
then, assuming that $u^{(1)}_\delta\in C^1(\overline{\Omega}\times
[0,T])$ is such that $\|U_0-u^{(1)}_\delta\|_{L^2(\Omega\times
[0,T])}\leq \delta$, we get, by using Lemma \ref{techlem2}, Lemma
\ref{lemstrongL2conver} (convergence of $u_\ve$ to $U_0$ in
$L^2(\Omega\times[0,T])$), continuity properties of $g(u,y)$ and
$g^\prime_u(u,y)$ in $u$ (conditions (\ref{4}), (\ref{5}),
(\ref{6})),  (\ref{rescaledpoincare}) and the second bound in
(\ref{uniformaprioriparabolicextended1}),
\begin{multline}
\lim_{\ve\to 0}\int_0^T\langle \C{G}_\ve(u_\ve),u_\ve-v_\ve
\rangle_\ve{\rm d} t= \lim_{\ve\to 0} \int_0^T\int_{S_\ve}\bigl(
g( u_\delta^{(1)},x/\ve)(u_\ve-v_\ve-\bar u_\ve+\bar v_\ve)
\bigr.\\
\bigl. + g^\prime_u(u_\delta^{(1)},x/\ve) (u_\delta^{(1)}-
V_0)(u_\ve-\bar u_\ve)\bigr) {\rm d}\sigma {\rm d}t  +O(\delta),
\label{pervoedeistvie111}
\end{multline}
provided that the limits exist. By using Proposition
\ref{convergtraces1} we identify the limits in the r.h.s. of
(\ref{pervoedeistvie111}) and then obtain
(\ref{thusconvergenceint}) by passing to the limit $\delta\to
0$.\hfill$\square$

\bigskip

Now, thanks to the monotonicity of the operator $\C{A}_\ve(u)$ we
can take $\liminf_{\ve\to 0}$ in (\ref{ravenstvoparab}) to obtain
by virtue of (\ref{firstsemifinal}), (\ref{secondsemifinal}),
(\ref{thusconvergenceint}) that
\begin{multline}
\int_{0}^{T} (\langle \partial_t U_0(t), U_0(t)- V_0(t)\rangle -
\langle f(t),U_0(t)-V_0(t)\rangle){\rm d}t\\
+\int_{0}^{T}\int_\Omega \int_{Y^*} a(D_xV_0+D_yV_1,y)\cdot
(D_xU_0+D_yU_1-D_xV_0-D_yV_1){\rm d}y{\rm d}x{\rm d} t\\
-\int_0^T M(U_0,U_1,V_0,V_1){\rm d} t\leq 0.
\label{varinequalityparab}
\end{multline}
This inequality is shown for any $V_0(x,t)\in
C^{\infty}(\overline{\Omega}\times [0,T])$ and any $V_1(x,y,t)\in
C^{\infty}(\overline{\Omega}\times \overline{Y}\times [0,T])$
($Y$-periodic in $y$), by an approximation argument it still holds
for any $V_0\in L^2(0,T;W^{1,2}(\Omega))$, $V_1\in
L^2(0,T;L^2(\Omega;W^{1,2}_{per}(Y)))$. Therefore we can set
$V_0=U_0$, $V_1=U_1\pm\delta\phi(x,t)w(y)$, where $w\in
W^{1,2}_{per}(Y)$, $\phi\in C^\infty(\overline{\Omega}\times
[0,T])$ and $\delta>0$ are arbitrary, divide
(\ref{varinequalityparab}) by $\delta$ and pass to the limit as
$\delta\to 0$ to get,
\begin{equation}
\int_0^T\int_\Omega\biggl(
 \int_{Y^*}  a(D_xU_0+D_yU_1,y)\cdot
D_yw{\rm d}y-\int_{S\cap Y} g(U_0,y)w{\rm d}\sigma_y \biggr)
\varphi(x,t) {\rm d}x{\rm d}t=0.
\end{equation}
This means, that $U_1$ solves (\ref{celleq}) with $u=U_0$ and
$\xi=D_x U_0$ for almost all $(x,t)\in\Omega\times [0,T]$. Now set
$V_0=U_0\pm\delta\Phi(x,t)$, $V_1=U_1$ , where $\Phi\in
C^\infty(\overline{\Omega}\times [0,T])$ and $\delta>0$ are
arbitrary, divide (\ref{varinequalityparab}) by $\delta$ and pass
to the limit as $\delta\to 0$. As a result we obtain
\begin{multline}
|Y^*|\int_{0}^{T} \int_\Omega \partial_tU_0(x,\tau) \Phi(x,\tau)
{\rm d}x{\rm d}\tau\\
+\int_{0}^{T}\int_\Omega (a^*(D_xU_0,U_0)\cdot
D_x\Phi-b^*(D_xU_0,U_0)\Phi-{\rm div}_x (g^*(U_0)\Phi)){\rm d}x
 {\rm d}\tau
\\
= |Y^*|\int_{0}^{T}\int_\Omega f(x,t)\Phi(x,\tau) {\rm d}x{\rm
d}\tau, \label{varequalityparab1}
\end{multline}
this yields (\ref{homogenizedeq2}).\hfill$\square$

\section{Properties of the homogenized problem}
\label{section7}

Define the operators ${\C A}^*, {\C B}^*, {\C T}^*: X\to X^*$ by
${\C B}^*(u)=b^*(Du,u)$,
$$
\langle{\C A}^*(u),v \rangle=\int_\Omega  a^*(Du,u)\cdot Dv {\rm
d}x,\ \forall v\in X,
$$
$$
\langle{\C T}^*(u),v\rangle
=\int_{\partial \Omega} g^*(u)\cdot
\nu\, v {\rm d}\sigma=\int_{\Omega} {\rm div}(g^*(u) v) {\rm d}x,
 \ \forall v\in X.
$$
Then, in terms of the operator $\C{F}^*(u)=\C{A}^*(u)-\C
{B}^*(u)-{\C T}^*(u)$, problems (\ref{homogenizedeq1}) and
(\ref{homogenizedeq2}) read
\begin{equation}
\C{F}^*(u)+\lambda u=f, \label{statopequation}
\end{equation}
\begin{equation}
\begin{cases}
\partial_t u+\C{F}^*(u)=f,\ t>0\\
u=\tilde u,\ \text{when}\ t=0.
\end{cases} \label{cauchyopequation}
\end{equation}
According to Theorem \ref{th1} there is a solution (obtained as
the limit of solutions of (\ref{eq1})) of (\ref{statopequation})
for every $f\in L^2(\Omega)$; similarly, by Theorem \ref{th2}
problem (\ref{cauchyopequation}) has a solution on the time
interval $[0,T]$ when  $f\in L^2(\Omega\times[0,T])$ and $\tilde
u\in L^2(\Omega)$. The solvability of problems
(\ref{statopequation}) and (\ref{cauchyopequation}) can be proved
for more general $f$, namely, we can assume merely $f\in X^*$ and
$f\in L^2(0,T;X^*)$ in (\ref{statopequation}) and
(\ref{cauchyopequation}), respectively. However we will focus
on the uniqueness results.

\bigskip

\noindent {\bf 1}({\it Properties of $a^*$ and $b^*$}). First we
show

\begin{lem}
\label{coefficients}
 The functions $a^*$ and $b^*$ given by (\ref{astar}), (\ref{bstar}) are
continuous. Moreover, there are constants $\gamma,\alpha,r>0$ and
$C$
such that
\begin{equation}
\label{boundsastar1} a^*(\xi,u)\cdot \xi\geq
\gamma|\xi|^2-C(|u|^2+1) \ \text{and}\ |a^*(\xi,u)|\leq
C(|\xi|+|u|+1),
\end{equation}
\begin{equation}
\label{boundsastar2} (a^*(\xi,u)-a^*(\zeta,v))\cdot(\xi-\zeta)\geq
\alpha |\xi-\zeta|^2-r(u-v)^2,
\end{equation}
$|b^*(\xi,u)|\leq C(|\xi|+|u|+1)$ and
\begin{multline}
\label{boundsbstar}  (b^*(\xi,u)-b^*(\zeta,v))(v-u)\leq
\frac{1}{4}(a^*(\xi,u)-a^*(\zeta,v))\cdot(\xi-\zeta))\\
+C\bigl(|u-v|^2+|u-v|^2(|\xi|+|u|+1)/(1+|u-v|)\bigr).
\end{multline}
\end{lem}

The proof of this Lemma is based on the study of properties of
solutions $w(y;\xi,u)$ of problem (\ref{celleq}). We will make use
of the following well-known results,
\begin{equation}
\int_{S\cap Y}\Bigl|w-\frac{1}{|Y^*|}\int_{Y^*} w{\rm
d}x\Bigr|^2\,{\rm d}\sigma\leq C \int_{Y^*}|Dw|^2\, {\rm d}x,
\label{poincare_per}
\end{equation}
\begin{equation}
 \int_{Y^*}|D_y w+\xi|^2\,{\rm d} y\geq \rho |\xi|^2,\ \rho >0,
\label{bounneumannhomogenization}
\end{equation}
for all $\xi\in \D{R}^N$, $w\in W^{1,2}_{per}(Y^*)$, where $C$ and
$\rho$ are independent of $w$ and $\xi$.



\begin{lem} \label{celllem} For any $\xi\in \D{R}^N$, $u\in\D{R}$ there is
a unique
(modulo an additive constant) solution $w(y;\xi,u)$ of problem
(\ref{celleq}) and we have
\begin{itemize}
\item[{\rm (a)}] $\displaystyle \int_{Y^*}|D_y w(y;\xi,u)|^2{\rm
d} y\leq C(|\xi|^2+|u|^2+1)$, \item[{\rm (b)}]
$a^*(\xi,u)\cdot\xi\geq \gamma|\xi|^2-C(|u|\,|\xi|+|u|^2+1)$ (with
$\gamma>0$), \item[{\rm(c)}] there are $\alpha,\, \beta
>0$ and $r$ such that, for any $\xi,\, \zeta\in\D{R}^N$ and $u,\,
v\in \D{R}$
$$ (a^*(\xi,u)-a^*(\zeta,v))\cdot(\xi-\zeta)\geq \alpha
|\xi-\zeta|^2-r(u-v)^2+\beta\int_{Y^*} |D\hat w|^2{\rm d} y,
$$
where $\hat w=w(y;\xi,u)-w(y;\zeta,v)$, \item[{\rm(d)}]
$w(y;\zeta,v)\to w(y;\xi,u)$ strongly in
$W^{1,2}_{per}(Y^*)\setminus \D{R}$ when $\zeta\to\xi, \ v\to u$.
\end{itemize}
\end{lem}

\begin{proof} The existence of a unique solution of (\ref{celleq})
in $W^{1,2}_{per}(Y^*)\setminus \D{R}$ easily follows from
assumptions (i)-(iii) and (vi) on the functions $a$ and $g$. To
show (a) we derive from (\ref{celleq}) by integrating by parts
\begin{equation}
\int_{Y^*} a(\xi+Dw,y)\cdot (\xi +Dw) {\rm d} y =\int_{S\cap Y}
g(u,y)w{\rm d}\sigma +\int_{Y^*} a(\xi+Dw,y)\cdot \xi {\rm d} y
\label{pochastyam}
\end{equation}
By applying the Poincar\'e inequality (\ref{poincare_per}) and
taking into account (\ref{7}), (\ref{4})   we obtain that for any
$k>0$,
\begin{multline}
\int_{Y^*} a(\xi+Dw,y)\cdot (\xi +Dw) {\rm d} y\leq
C(|u|+1)\|Dw\|_{L^2(Y^*)}+C|\xi|\,\|\xi+Dw\|_{L^2(Y^*)}
\\
\leq C(|u|+1)(\|\xi+Dw\|_{L^2(Y^*)}+|\xi|) +C|\xi|
\,\|\xi+Dw\|_{L^2(Y^*)}
\\
\leq k((|u|+1)^2+\frac{C}{k}(|\xi|^2+\|\xi+Dw\|^2_{L^2(Y^*)}),
\label{finbound}
\end{multline}
where $C$ is independent of $k$, $u$ and $\xi$. If we choose $k$
in (\ref{finbound}) large enough and use (\ref{3}) we get
$$
\int_{Y^*} |\xi+Dw|^2{\rm d} y\leq C (|u|^2+|\xi|^2+1),
$$
that in turn implies (a).

By using (\ref{bounneumannhomogenization}) on the l.h.s. of
(\ref{pochastyam}) and (\ref{poincare_per}) in conjunction with
(\ref{4}), (\ref{7}) in the first term of the r.h.s., we easily
derive (b).

In order to show (c) we  use (\ref{celleq}) to get by integrating
by parts
\begin{multline}
(a^*(\xi,u)-a^*(\zeta,v))\cdot(\xi-\zeta)=
 \int_{S\cap Y} (g(v,y)-g(u,y))\hat w {\rm d} \sigma
 \\
 +
 \int_{Y^*}
 (a(\xi+D_y w(y;\xi,u))-a(\zeta+D_y w(y;\zeta,v)))
 \cdot
 (\xi-\zeta +D_y \hat w){\rm d} y.
 \label{momotonL_1a}
 \end{multline}
Taking into account (\ref{5}), (\ref{7}) and applying
(\ref{poincare_per}) we can estimate the first term $I_1$ on the
r.h.s. of (\ref{momotonL_1a}) as
\begin{equation}
|I_1|\leq k|u-v|^2+\frac{C}{k}\int_{Y^*} |D\hat w|^2{\rm d} y, \
\text{for any} \ r>0, \label{promezhdeistvie}
\end{equation}
where $C$ is independent of $k$, $\xi$, $\zeta$, $u$, $v$. In view
of (\ref{monoton}) and  (\ref{bounneumannhomogenization}) we have
the following lower bound for the second term $I_2$ in
(\ref{momotonL_1a})
$$
I_2\geq (1-\delta)\kappa \rho|\xi-\zeta|^2+\delta\kappa\int_{Y^*}
|\xi-\zeta+D_y \hat w|^2{\rm d} y
$$
with $0<\delta<1$ to be chosen later. On the other hand, by the
elementary inequality $a^2\leq 2(a+b)^2+2b^2$,
$$
\int_{Y^*}|D_y \hat w|^2{\rm d} y\leq 2 \int_{Y^*}|\xi-\zeta+D_y
\hat w|^2{\rm d} y+2|\xi-\zeta|^2,
$$
thus
$$
I_2\geq \kappa(\rho-\delta(\rho+1))
|\xi-\zeta|^2+\frac{\delta\kappa}{2}\int_{Y^*} |D_y \hat w|^2{\rm
d} y.
$$
Choose $0<\delta<1$ so that $\rho-\delta(\ae+1)>0$ and set $k=4 C
/(\delta\kappa)$ (where $C$ is the constant appearing in
(\ref{promezhdeistvie})), we thus obtain (b) with
$\alpha=\kappa(\rho-\delta(\rho+1))>0$,
$\beta=(\delta\kappa)/4>0$.

Finally, statement (d) is a direct consequence of (a) and (c).
\end{proof}

\noindent {\it Proof of Lemma} \ref{coefficients}. According to
Lemma \ref{celllem} it suffices only to show (\ref{boundsbstar}).
Set $\hat w=w(y;\xi,u)-w(y;\zeta,v)$, we have by using
(\ref{poincare_per}) and assumptions (i), (iii), (iv) on $g$,
\begin{multline}
\label{boundBstar} (b^*(\xi,u)-b^*(\zeta,v))(v-u)= (v-u)
\int_{S\cap Y}
g^\prime_u(v,y)\hat w {\rm d} \sigma_y \\
+ (v-u) \int_{S\cap Y} (g^\prime_u(u,y)-g^\prime_u(v,y))w(y;\xi,u)
{\rm d} \sigma_y
\\
\leq C|u-v|\|D\hat w\|_{L^2(Y^*)} +
C|u-v|^2\|Dw(\,\cdot\,;\xi,u)\|_{L^2(Y^*)}/(1+|u|+|v|).
\end{multline}
Then statements (a) and (c) of  Lemma \ref{celllem} yield
(\ref{boundsbstar}).\hfill $\square$

\begin{rem}
\label{aremarka} In the case when the function $g(u,y)$ is linear
in $u$, bound (\ref{boundsbstar}) simplifies to the following one,
\begin{equation*}
 (b^*(\xi,u)-b^*(\zeta,v))(v-u)\leq
\frac{1}{4}(a^*(\xi,u)-a^*(\zeta,v))\cdot(\xi-\zeta)) +C|u-v|^2.
\end{equation*}
\end{rem}

Let us consider next the particular case when $a(\xi,y)$ is linear
in $\xi$, i.e.  $a$ is given by $a(\xi,y)=A(y)\xi$ with $A\in
L^\infty(Y;\D{R}^{N\times N})$, $A(y)\xi\cdot\xi\geq \kappa
|\xi|^2$ ($\kappa>0$), $\forall \xi\in\D{R}^N,\, y\in Y$. Then we
can write the solution of (\ref{celleq}) as the sum
$w(y;\xi,u)=w^{(1)}(y;\xi)+\tilde w(y;u)$ with $w^{(1)}$ solving
(\ref{celleq}) and $\tilde w$ being a unique (up to an additive
constant) solution of
\begin{equation}
\begin{cases}
{\rm div}\, (A(y)D_y \tilde w)=0\ \text{in}\ Y^*\\
A(y)D_y \tilde w\cdot\nu=g(u,y)\ \text{on}\ S\cap Y\\
\tilde w\ \text{is $Y$-periodic}.
\end{cases}
\label{celleq1tilde}
\end{equation}
Note that $w^{(1)}(y;\xi)$ depends linearly on $\xi$, also we have

$\|\tilde w(y;u)\|_{W^{1,2}(Y^*)\setminus \D{R}}\leq C(|u|+1)$,
$\|\tilde w(y;u)-\tilde w(y;v\|_{W^{1,2}(Y^*)\setminus \D{R}}\leq
C|u-v|$,

$\|\tilde w^\prime_u(y;u)-\tilde
w^\prime_u(y;v)\|_{W^{1,2)}(Y^*)\setminus \D{R}} \leq
C|u-v|/(1+|u|+|v|)$,

\noindent where $C$ is independent of $u$, $v$. The proof of these
bounds is analogous to that of (\ref{Theta1}) - (\ref{Theta3}).
Thus we have
\begin{multline}
\label{aislinear}
b^*(\xi,u)=\frac{\partial}{\partial u}
\int_{Y^*} A(y)D_y\tilde w(y;u) \cdot D_y w^{(1)}(y;\xi) {\rm d} y
\\
+ \int_{Y^*} A(y)D_y\tilde w^\prime_u(y;u) \cdot D_y \tilde w(y;u)
{\rm d} y = H^{\prime}(u)\cdot \xi+h(u)
\end{multline}
with $H$, $h$ such that $|H(u)-H(v)|\leq C|u-v|$, $|h(u)-h(v)|\leq
C|u-v|$.

\bigskip

\noindent {\bf 2}({\it Uniqueness results for problem}
(\ref{statopequation})). In the particular cases when the
dimension of the space $N\leq 3$ or $a(\xi,y)$ is linear in $\xi$
or $g(u,y)$ is linear in $u$ we show that problem
(\ref{statopequation}) cannot have two distinct solutions for
sufficiently large $\lambda$.

The following inequality will be used to estimate the expressions
involving traces on $\partial \Omega$. For every $\delta>0$ there
is $\Lambda_\delta$ such that
\begin{equation}
\label{ineqfortraces} \int_{\partial\Omega} |w|^2 {\rm d}\sigma
\leq\delta
\|Dw\|^2_{L^2(\Omega)}+\Lambda_\delta\|w\|_{L^2(\Omega)},\ \forall
w\in W^{1,2}(\Omega).
\end{equation}
This inequality is a consequence of the compactness of the trace
operator $T_{\partial\Omega}: W^{1,2}(\Omega)\to
L^2(\partial\Omega)$, $T_{\partial\Omega}u=\text{trace of}\ u\
\text{on}\
\partial\Omega$. Thanks to the Lipschitz continuity of $g(u,y)$ in the variable $u$,
inequality (\ref{ineqfortraces}) implies that
\begin{equation}
\label{boundTstar}
 |\langle \C{T}^*(u)-\C{T}^*(v), u-v
\rangle|\leq \frac{\alpha}{4}\|u-v\|_{X}^2 +C
\|u-v\|^2_{L^2(\Omega)},
\end{equation}
where $\alpha>0$ is the same as in (\ref{boundsastar2}).

Let $u$, $v$ be solutions of (\ref{statopequation}).

\smallskip

\noindent {\bf Case I} ({\it $g(u,y)$ is linear in $u$}). By using
Lemma \ref{coefficients}, Remark \ref{aremarka} and
(\ref{boundTstar}) we get
\begin{equation}
\label{finamonoton} \langle
\C{F}^*(u)-\C{F}^*(v)+\lambda(u-v),u-v\rangle\geq
\frac{\alpha}{4}\|u-v\|^2_{X}+(\lambda
-\hat\lambda_0)\|u-v\|^2_{L^2(\Omega)},
\end{equation}
with $\hat\lambda_0$ independent of $\lambda$. It follows that
$u=v$ if $\lambda\geq\hat\lambda_0$.

\noindent {\bf Case II} ({\it $a(\xi,y)$ is linear in $\xi$}). We
have, according to (\ref{aislinear}),
\begin{multline*}
\langle {B}^*(u)-\C {B}^*(v)),v-u\rangle= |Y^*|\int_\Omega
(u-v)({\rm div} (H(u)-H(v))+h(u)-h(v)) {\rm d}x
\\
=|Y^*|\int_\Omega (D(v-u)\cdot(H(u)-H(v))+(u-v)(h(u)-h(v))) {\rm
d}x
\\+|Y^*|\int_{\partial \Omega}(u-v)(H(u)-H(v))\cdot\nu\,{\rm d}\sigma
\\
\leq \frac{\alpha}{4}\|u-v\|^2_{X}+ C\|u-v\|^2_{L^2(\Omega)},
\end{multline*}
where we have used (\ref{ineqfortraces}). This inequality and
Lemma \ref{coefficients} yield (\ref{finamonoton}) (with possibly
another constant $\hat \lambda_0$).

\noindent {\bf Case III} ({\it The space dimension $N$ is two or
three}). It is well known that for these space dimensions
$X(=W^{1,2}(\Omega))$ is compactly embedded into $L^4(\Omega)$,
moreover $\|w\|_{L^4(\Omega)}^2\leq
C\delta\|w\|_{X}^2+C\delta^{-N/(4-N)}\|w\|_{L^2(\Omega)}^2$ for
all $w\in X$ and $\delta>0$, where $C$ is independent of
$\delta>0$ and $w$ (see, e.g., \cite{L}). By using this
inequality, Lemma \ref{coefficients} and (\ref{boundTstar}) we
easily show that
\begin{multline}
\langle \C{F}^*(u)-\C
{F}^*(v),u-v\rangle\geq \frac{\alpha}{4}\|u-v\|^2_{X} \\
-C(\delta
\|u-v\|^2_{X}+\delta^{-N/(4-N)}\|u-v\|^2_{L^2(\Omega)})(\|u\|_{X}+1),\
\forall \delta>0.
\label{H2monoton}
\end{multline}
On the other hand Lemma \ref{coefficients} and the very definition
of $\C{T}^*(u)$ imply that for every $w\in X$ $\langle
\C{A}^*(w),w\rangle\geq
\gamma\|w\|_{X}^2-C(\|w\|^2_{L^2(\Omega)}+1)$, $|\langle
\C{B}^*(w),w\rangle|\leq
C(\|w\|_{X}+\|w\|_{L^2(\Omega)}+1)\|w\|_{L^2(\Omega)}$ and
$|\langle \C{T}^*(w),w\rangle|\leq C\|w\|_{X}\|w\|_{L^2(\Omega)}$.
Therefore there is $\tilde\lambda_0$ such that $\langle
{F}^*(u),u\rangle\geq\frac{\gamma}{2}\|u\|^2_{X}-\tilde\lambda_0
\langle u, u\rangle$, hence, for $\lambda\geq \tilde\lambda_0$ we
have the a-priori bound $\|u\|_X\leq C(\|f\|_{X^*}+1)$ with $C$
independent of $u$, $f$ and $\lambda\geq\tilde\lambda_0$. Thus,
$u$ and $v$ being solutions of (\ref{statopequation}), estimate
(\ref{H2monoton}) yields
$$
\frac{\alpha}{4}\|u-v\|^2_{X}+\lambda \|u-v\|^2_{L^2(\Omega)} \leq
C(\|f\|_{X^*}+1)(\delta
\|u-v\|^2_{X}+\delta^{-N/(4-N)}\|u-v\|^2_{L^2(\Omega)}),
$$
and by setting $\delta=\alpha/(8 C((\|f\|_{X^*}+2))$ we get $u=v$
as far as $\lambda\geq
\hat\lambda_0$($=\max\{\tilde\lambda_0,C(\|f\|_{X^*}+1)\delta^{-N/(4-N)}\}$).
($\hat\lambda_0$ can be chosen independent of $f$ if $N=2$.)

\bigskip

\noindent {\bf 2}({\it Uniqueness results for problem}
(\ref{cauchyopequation})). Given $T>0$, we show that problem
(\ref{cauchyopequation}) cannot have two distinct solutions $u$,
$v$ on the time interval $[0,T]$ if  $a(\xi,y)$ is linear in $\xi$
or $g(u,y)$ is linear in $u$. Indeed, $w=u-v$ satisfies $
\partial_t
\langle w(t),w(t)\rangle+ 2\langle \C{F}^*(u(t))-\C
{F}^*(v(t)),u(t)-v(t)\rangle=0$, $0<t<T$, and $w(0)=0$, while
(\ref{finamonoton}) yields $-2\langle \C{F}^*(u(t))-\C
{F}^*(v(t)),u(t)-v(t)\rangle\leq C\langle w(t),w(t)\rangle$,
$0<t<T$, therefore $e^{-Ct}\|w(t)\|^2_{L^2(\Omega)}\leq 0$ so that
$w\equiv 0$.

In the case when space dimension is two we also have the
uniqueness result. Note that we have at least one solution $u\in
L^2(0,T;X)$ of (\ref{cauchyopequation}). Then, if $v$ is another
solution we set $w=u-v$, $R(t)=\langle w(t),w(t)\rangle$, and
derive by using (\ref{H2monoton}) with $\delta=\alpha/(8
C((\|u\|_{X}+1))$,
$$
R^\prime(t)-CR(t)(\|u(t)\|_X+1)^2\leq 0,\ 0<t<T, \ \text{and}\
R(0)=0.
$$
This implies that $R(t){\rm
exp}\{-C\int_0^t(\|u(\tau)\|_X+1)^2\,{\rm d}\tau\}\leq 0$ and
therefore $R\equiv 0$, i.e. $u=v$.

{\bf Acknowledgments.} The work of V.Rybalko is partially
supported by the Grant of NASU for Young Scientists. This work was
originated and partially done when V.Rybalko enjoyed the
hospitality of the Narvik University College whose support is gratefully acknowledged.


\end{document}